\documentclass[12pt]{article}

\usepackage{amsmath}
\usepackage{amssymb}
\usepackage{amsthm}
\usepackage{xcolor}

\usepackage{ifpdf}
\ifpdf
\usepackage[colorlinks=true,linkcolor=black,final,hyperindex,citecolor=black]{hyperref}
\else
\usepackage[colorlinks,final,hyperindex,hypertex]{hyperref}
\fi

\usepackage{marginnote}

\usepackage{enumitem}  % http://ctan.org/pkg/enumitem
\usepackage{calc}% http://ctan.org/pkg/calc
\setlist{labelindent=1pt,itemsep=.5em}
\setlist[itemize]{leftmargin=1.2cm}
\setlist[enumerate]{itemindent=0em,leftmargin=1.2cm}
\usepackage{cite}

\usepackage{marginnote}

\allowdisplaybreaks

\newtheorem{definition}{Definition}
\newtheorem{theorem}{Theorem}
\newtheorem{lemma}{Lemma}
\newtheorem{corollary}{Corollary}
\newtheorem{remark}{Remark}

\newtheorem{example}{Example}

\usepackage{authblk}

\title{Lower and upper bounds of joint $(f,\delta)$-numerical radius functions}

\author[1]{Zameddin I. Ismailov\thanks{E-mail: \texttt{zameddin.ismailov@gmail.com}}}
\author[2]{Sergei Silvestrov\thanks{E-mail: \texttt{sergei.silvestrov@mdu.se} (Corresponding author)}}
\author[1]{Pembe Ipek Al\thanks{E-mail: \texttt{ipekpembe@gmail.com}}}
\affil[1]{Department of Mathematics, Karadeniz Technical University, Trabzon, Türkiye}
\affil[2]{M\"{a}lardalen University, Department of Business and Mathematics\\
Faculty of Philosophy, Box 883, 72123 V\"{a}ster{\aa}s, Sweden}
\date{} % no date; remove this line if you want today's date

\begin{document}
%\linenumbers % Turn on line numbers with included \usepackage{lineno}

\maketitle

\begin{abstract}
In this study, the classical results on the joint numerical radius for $n$-tuples of Hilbert space operators are extended to the setting of the joint $(f,\delta)$-numerical radius. New and diverse contributions to this area are provided, including novel estimates for the lower and upper bounds of the $(f,\delta)$-numerical radius in the context of sectorial operators.
\end{abstract}

\tableofcontents

%----------------------------------------------------------------------------------------------------------------------
\smallskip

\smallskip

\section{Introduction}
\label{Sec:1}

The concepts of numerical range and numerical radius of an operator play a
fundamental role in both pure and applied mathematics, and have been extensively
studied due to their wide-ranging applications in fields such as engineering,
quantum computing, quantum physics, numerical analysis, differential equations,
fluid dynamics and the geometry of Banach spaces (see \cite{Axelsson1993,Bonsal1971,Horn1991}).
The investigation of numerical radius inequalities and their various refinements
has become a prominent and contemporary research trend in the theory of operators
on Hilbert spaces.

Throughout this article, $\mathcal{H}$ denotes a complex Hilbert space endowed with the inner
product $\langle \cdot, \cdot\rangle $ and associated norm $\|\cdot\|$. Let $\mathfrak{B}(\mathcal{H})$ be the $C^{*}$-algebra of all bounded linear operators acting on $\mathcal{H}$ and $S_{1}(\mathcal{H})$ be the unit sphere of the Hilbert space $\mathcal{H}$.
The numerical
radius of an operator $A$ is defined by
\[
    \omega(A) = \sup \{ |\langle Ax,x\rangle| : x \in S_{1}(\mathcal{H}) \}.
\]
The usual operator norm of an operator $A \in \mathfrak{B}(\mathcal{H})$
is defined by
\[
    \|A\| = \sup \{ \|Ax\| : x \in S_{1}(\mathcal{H}) \}.
\]

It is well recognized in operator theory literature that the exact computation of the numerical radius $\omega(A)$ for a general operator $A \in \mathfrak{B}(\mathcal{H})$ is a highly challenging problem.
Explicit formulas are typically available only for particular classes of operators that exhibit special structural properties.
Consequently, a significant line of research has been devoted to the development of lower and upper bounds for $\omega(A)$.

Recall that the numerical radius $\omega(\cdot)$ is a norm on $\mathfrak{B}(\mathcal{H})$, and among the most basic bounds for $\omega(A)$ is the following two-sided inequality \cite{Halmos1982},
\begin{equation}
\label{equ1}
    \frac{1}{2}\|A\| \leq \omega(A) \leq \|A\|.
\end{equation}
The right-hand side of \eqref{equ1} becomes an equality when $A$ is normal, while the left-hand side becomes an equality when $A^{2} = 0,$ where $0$ is the zero operator in $\mathfrak{B}(\mathcal{H})$.

Sharpening inequality \eqref{equ1} has received considerable attention in the literature. Among these attempts, the following two-sided inequality is of interest \cite{Kittaneh2005}:
\begin{equation}
\label{equ2}
    \frac{1}{4}\left\|A^{*}A + AA^{*}\right\|
    \leq \omega^{2}(A)
    \leq \frac{1}{2}\left\|A^{*}A + AA^{*}\right\|.
\end{equation}
Inequality \eqref{equ2} refines both the left and right inequalities in \eqref{equ1}.

For our purpose, we will need the following definition for $n$-tuple of operatorss in $   \mathfrak{B}^{n}(\mathcal{H}) := \mathfrak{B}(\mathcal{H}) \times \cdots \times
\mathfrak{B}(\mathcal{H}) \ (n \ times), n\geq 1.$

\begin{definition}
\label{def1}
Let $\mathbb{A} = (A_{1}, \ldots, A_{n})$, $\mathbb{B} = (B_{1}, \ldots, B_{n}) \in \mathfrak{B}^{n}(\mathcal{H})$ be two $n$-tuples of operators, and $\lambda \in \mathbb{C}$. The following definitions are introduced.
\begin{enumerate}[label=\upshape{\arabic*.}, ref=\upshape{\arabic*}, labelindent=5pt, leftmargin=*]
    \item $\lambda \mathbb{A} = (\lambda A_{1}, \ldots, \lambda A_{n})$;
    \item $\mathbb{A} + \mathbb{B} = (A_{1} + B_{1}, \ldots, A_{n} + B_{n})$;
    \item $\mathbb{A}^{*} = (A_{1}^{*}, \ldots, A_{n}^{*})$, where $ A_{m}^{*} $ is the adjoint of the operator $ A_{m} $ in $ \mathcal{H}, $ $ 1\leq m \leq n; $
    \item $\mathbb{A} = 0$ if and only if for each $1 \le m \le n$, $A_{m} = 0$;
    \item $Re \mathbb{A} = (Re A_{1}, \ldots, Re A_{n})$
    and $Im \mathbb{A} = (Im A_{1}, \ldots, Im A_{n}) ;$
    \item $ \vert \mathbb{A} \vert = \left( \vert A_{1} \vert , \ldots , \vert A_{n} \vert \right),   $ where $ \vert A_{m} \vert = (A_{m}^{*}A_{m})^{\frac{1}{2}} , \  1\leq m \leq n; $
    \item $\mathbb{A} x =(A_{1}x,\ldots,A_{n}x)\in \mathcal{H}^{n}$ and     $\Vert\mathbb{A} x\Vert =
    \left( \sum\limits_{m=1}^{n} \Vert A_{m}x\Vert^{2} \right)^{\frac{1}{2}}, \ x \in \mathcal{H}$;
    \item The operator norm of $\mathbb{A} \in \mathfrak{B}^{n}(\mathcal{H})$ is defined as
    \[
        \|\mathbb{A}\| =
        \sup \left\{
            \left( \sum\limits_{m=1}^{n} \|A_{m}x\|^{2} \right)^{\frac{1}{2}}
            : x \in S_{1}(\mathcal{H})
        \right\} =\left\| \sum\limits_{m=1}^n A_m^* A_m \right\|^{1/2} ;
    \]
    \item $ \langle \mathbb{A}x,x \rangle = \left( \langle A_{1}x,x\rangle, \ldots , \langle A_{n}x,x\rangle \right) \in \mathbb{C}^{n} , \   x\in \mathcal{H}; $
    \item $|\langle \mathbb{A} x,x\rangle| =
    \left( \sum\limits_{m=1}^{n} |\langle A_{m}x, x\rangle|^{2} \right)^{\frac{1}{2}}, \ x \in \mathcal{H}$;
    \item The joint numerical range of $\mathbb{A} \in \mathfrak{B}^{n}(\mathcal{H})$ is a subset of $\mathbb{C}^{n}$ defined by
    $
        W(\mathbb{A}) =
        \left\{
            \bigl(\langle A_{1}x, x\rangle, \ldots, \langle A_{n}x, x\rangle\bigr)
            : x \in S_{1}(\mathcal{H})
        \right\}.\ \textnormal{(See  \textnormal{\cite{Jana2024,Muller2020,Takaguchi1981}})}
   $
    \item The joint numerical radius of $\mathbb{A} \in \mathfrak{B}^{n}(\mathcal{H})$ is defined by
    \[
        \omega(\mathbb{A}) =
        \sup \left\{
            \left( \sum\limits_{m=1}^{n} |\langle A_{m}x, x\rangle|^{2} \right)^{\frac{1}{2}}
            : x \in S_{1}(\mathcal{H})
        \right\}.
    \]
\end{enumerate}
\end{definition}
Clearly, for $\mathbb{A} \in \mathfrak{B}^{n}(\mathcal{H})$,
$   \omega(\mathbb{A}) \leq
    \Vert \mathbb{A} \Vert. $

In 2014, Drnov\v{s}ek and Müller \cite{Drnovsek2014}, and Müller \cite{Muller2014} discussed the joint numerical radius of an $n$-tuple of operators and studied the lower bound of the numerical radius. More precisely, for $\mathbb{A} \in \mathfrak{B}^{n}(\mathcal{H})$,
\(
    \omega(\mathbb{A}) \geq
    \dfrac{\Vert \mathbb{A} \Vert}{2\sqrt{n}}.
\)
Therefore, the joint numerical radius is a norm equivalent to the operator norm on
$\mathfrak{B}^{n}(\mathcal{H})$. In particular, if $ \mathbb{A} \in \mathfrak{B}^{n}(\mathcal{H})$
is an $n$-tuple of isometries on a Hilbert space $\mathcal{H}$, then $ \Vert \mathbb{A} \Vert =  \|\sum\limits_{j=1}^n I \|^{1/2} = \|nI\|^{1/2} = \sqrt{n}. $
Thus, for any $n$-tuple of isometries, $\omega(\mathbb{A}) \leq \sqrt{n}$,
that is, the numerical radius is bounded from above by $\sqrt{n}$.

For any commuting $n$-tuples in general, $r(\mathbb{A})\leq \omega(\mathbb{A}) \leq \Vert \mathbb{A} \Vert,$ where $r(\mathbb{A})$ is the joint spectral radius.
The submultiplicativity  $\omega(TS) \le \omega(T)\omega(S)$ is generally false.
However, for commuting operators $T$ and $S$ the bound $ \omega(TS) \le 2\omega(T)\omega(S) $ holds.

While the numerical range $W(A)=\{\langle Ax,x\rangle: x\in \mathcal{H}, \|x\|=1\}$ of any bounded operator $T$ on a separable complex Hilbert space is always a convex set,
the joint numerical range of $n$-tuple $W(\mathbb{A})$ is, in general, not convex for $n\ge 2$.
For a single normal operator, the closure of its numerical range is the convex hull of its spectrum.
The joint numerical range of a commuting tuple of normal operators is always a convex set, and the closure of the joint numerical range
equals the convex hull of the joint approximate spectrum
\begin{multline*}
 \sigma_{ap}(A)=\{\lambda = (\lambda_1, \dots, \lambda_n) \in \mathbb{C}^n: \exists \{x_k, \Vert x_k \Vert=1\}_{k=0}^{\infty} \subset \mathcal{H}: \\
 \lim\limits_{k \to \infty} \|(A_j - \lambda_j I)x_k\| = 0 \quad \text{for all } j = 1, \dots, n \}.
\end{multline*}
For commuting normal operators, the joint numerical range's closure also coincides with its ``joint numerical status," (joint ``algebraic numerical range") which is the set of all states
\begin{multline*}
\mathcal{S}(C^*(\mathbb{A}, I))=\{\phi: \mathcal{A} \to \mathbb{C}: \phi \text{ is a linear functional},  \\
\forall\ T \in \mathcal{A} \phi(T^*T) \geq 0 \text{ (Positivity)},
 \phi(I) = 1 \text{ (Normalization)}
\end{multline*}
on the
$C^*$-algebra generated by the operators $C^*(\mathbb{A}, I)$ (the smallest $C^*$-subalgebra of $B(\mathcal{H})$ containing all operators of the tuple $\mathbb{A}$ and the identity operator $I$), which is
a commutative $C^*$-algebra, since  $A_1, \dots, A_n$ are pairwise commuting and normal.
This set is always convex and compact. Because the states of a
$C^*$-algebra are the weak-star closure of the convex hull of its ``pure states" (which correspond to vector states
in the irreducible case), this leads to a ``convex hull" structure.

If the commuting isometries are also normal (unitary operators, $U^*U=UU^*=I$), then the numerical radius is exactly equal to the operator norm.
When the isometries in $n$-tuple $\mathbb{A}$ commute, the joint numerical radius is equal to the joint spectral radius.
For commuting normal isometries (unitaries),  $ \omega(\mathbb{A}) = \|\mathbb{A}\| = \sqrt{n} $, as a consequence of the existence of a joint spectrum in which the values can be simultaneously maximized.
The Spectral Mapping Theorem for commuting tuples forces the numerical radius to hit the upper bound of the norm when the operators commute and behave like unitaries on their joint spectrum.

The numerical radius drops significantly when the isometries anticommute ($A_i A_j = -A_j A_i$).
For anticommuting isometries (e.g., generators of a Clifford algebra), the joint numerical radius is significantly smaller than $\sqrt{n}$.
For anticommuting isometries, the bound often reduces toward $\omega(\mathbb(A))\leq \sqrt{\frac{n}{2}} $
or even lower, depending on the dimension of the space.
For the $n$-tuple of anticommuting isometries, the generators of a Clifford algebra $C\ell_n(\mathbb{C})$), the joint numerical range is effectively a unit ball in $\mathbb{R}^n$, yielding
$ \omega(\mathbf{A}) = 1$.
The norm $\|\mathbb{A}\|$ grows as $\sqrt{n}$, whereas the numerical radius remains at $1$.
Since $\frac{w(\mathbf{A})}{\|\mathbf{A}\|} = \frac{1}{\sqrt{n}}\rightarrow 0, n\rightarrow +\infty$, for highly non-commuting systems, the numerical radius becomes a negligible fraction of the total operator norm.
The ratio (or the ``gap") $\frac{w(\mathbf{A})}{\|\mathbf{A}\|} = \frac{1}{\sqrt{n}}$  may be interpreted as a quantitative measure of the non-commutativity.
\begin{example}[Pauli tuple]
For the 3-tuple \[\mathbb{\sigma} = (\sigma_x=\begin{pmatrix}
  0 & 1 \\
  1 & 0
\end{pmatrix}, \sigma_y=\begin{pmatrix}
 0 & -i \\
  i & 0
\end{pmatrix}, \sigma_z=\begin{pmatrix}
  1 & 0 \\
  0 & -1
\end{pmatrix}),\]
the operators are unitary isometries, $\sigma_j^*\sigma_j=\sigma_j\sigma_j^*=I$ (introduced by Pauli to describe electron spin), which anticommute $\sigma_j\sigma_k = - \sigma_k\sigma_j$  for $j\neq k$.
For the $3$-tuple $(\sigma_x, \sigma_y, \sigma_z)$,
\begin{align*}
& \|\mathbb{A}\| = \|\sqrt{\sigma_x^*\sigma_x+\sigma_y^*\sigma_y+\sigma_z^*\sigma_z}\|=\|\sqrt{I+I+I}\| = \sqrt{3}, \\
& \omega(\mathbb{A}) =
        \sup\limits_{v \in S_{1}(\mathcal{H})} \left\{
            \sqrt{\left(|\langle \sigma_x v, v\rangle|^{2} +|\langle \sigma_y v, v\rangle|^{2} +|\langle \sigma_z v, v\rangle|^{2} \right)}
        \right\}=1
\end{align*}
since for $v = \begin{pmatrix} \alpha \\ \beta \end{pmatrix} \in \mathbb{C}^2$ such that $\Vert v \Vert=\sqrt{|\alpha|^2+|\beta|^2}=1$,
if $\alpha =Re(\alpha)+iIm(\alpha)=a_1 + ia_2 $ and $\beta =Re(\beta)+iIm(\beta)=b_1 + ib_2$, $a_1, a_2, b_1, b_2 \in \mathbb{R}$, then
\begin{align*}
 &  \langle \sigma_x v, v \rangle = \bar{\alpha}\beta + \alpha\bar{\beta} = 2(a_1b_1 + a_2b_2), \\
& \langle \sigma_y v, v \rangle = -i(\bar{\alpha}\beta - \alpha\bar{\beta}) = 2(a_2b_1-a_1b_2), \\
 &   \langle \sigma_z v, v \rangle = |\alpha|^2 - |\beta|^2 = (a_1^2 + a_2^2) - (b_1^2 + b_2^2), \\
& |\langle\sigma_x v, v\rangle|^{2} + |\langle \sigma_y v, v\rangle|^{2} + |\langle \sigma_z v, v\rangle|^{2} \\
& \quad =4(a_1b_1 + a_2b_2)^2+4(a_2b_1-a_1b_2)^2 + ((a_1^2 + a_2^2) - (b_1^2 + b_2^2))\\
& \quad =(a_1^2 + a_2^2 + b_1^2 + b_2^2)^2=(|\alpha|^2+|\beta|^2)^2 = 1.
\end{align*}
Thus, the vectors $(\langle \sigma_x v, v\rangle, \langle \sigma_y v, v\rangle, \langle \sigma_z v, v\rangle)$ are points on the Bloch sphere (Euclidean length $1$).
These vectors cover the entire unit sphere in $\mathbb{R}^3$ when $a_1, a_2, b_1, b_2 \in \mathbb{R}, a_1^2 + a_2^2 + b_1^2 + b_2^2=|\alpha|^2+|\beta|^2)^2 = \Vert v \Vert^2= 1$.
The joint numerical range of any tuple of matrices is convex. The joint numerical range $W(\sigma_x, \sigma_y, \sigma_z)$ is the unit ball in $\mathbb{R}^3$ (Bloch ball), which is
the smallest convex set containing the unit sphere in $\mathbb{R}^3$.
For the $2$-tuple $\mathbb{A}=(\sigma_x, \sigma_y)$,
\begin{align*}
\|\mathbb{A}\| & = \|\sqrt{\sigma_x^*\sigma_x+\sigma_y^*\sigma_y}\|=\|\sqrt{I+I}\| = \sqrt{2},\\
w(\mathbb{A})& = \sup\limits_{\Vert v \Vert^2= a_1^2+a_2^2+b_1^2+b_2^2=1} 4 (a_1^2+a_2^2) (b_1^2+b_2^2)
=\sup\limits_{|\alpha|^2+|\beta|^2=1} 4|\alpha|^2  |\beta|^2 =1,
\end{align*}
with the supremum achieved for $\alpha=\beta=\frac{1}{2}$. The vectors $(\langle \sigma_x v, v \rangle,\langle \sigma_y v, v \rangle)$ cover the entire unit disc in $\mathbb{R}^2$ for $v, \Vert v \Vert= 1$.
Hence, the joint range is a unit disk in $\mathbb{R}^2$.

The drop from \(w(\mathbb{A}) = \sqrt{n}\)  for commuting isometries to \(w(\mathbb{A}) = 1\) for anticommuting isometries
is at the heart of the mathematical foundation of the uncertainty principle in Quantum Mechanics.
%since
%for anticommuting operators $\{A, B\} = 0$, the Robertson-Schrödinger relation implies $\Delta A^2 + \Delta B^2 \ge 2 - (\langle A \rangle^2 + \langle B \rangle^2)$, and
%since $\sum\limits \langle A_j \rangle^2 \le w(\mathbb{A})^2 = 1$, the sum of variances is strictly bounded away from zero.
\end{example}

The theory of joint numerical range is an active field of research in operator theory,
and various important results including geometric structure of joint numerical range
have been obtained in the last few years (see \cite{Drnovsek2014} and \cite{Muller2014}).

General information on the joint numerical radius can be found in \cite{Bhunia2022}, \cite{Dragomir2013}, \cite{Gau2021}. Some general information and results on the joint numerical ranges and numerical radii can be
found in \cite{Feki2024, Gau2021, Guesba2024, Guesba2025, Mal2023, Muller2020, Wang2025} and references therein.

We now define another type of numerical range, numerical radius, and operator norm will be given.

\begin{definition}[\!\!\cite{Stampfli1970}]
\label{def2}
Let $A \in \mathfrak{B}(\mathcal{H})$ and $0 \leq \delta \leq \|A\|$.
The $\delta$-numerical range of $A$ is defined as
\(
    W_{\delta}(A)=cl\left( \left\lbrace  \langle Ax,x\rangle: x\in S_{1}(\mathcal{H}), \ \Vert Ax \Vert\geq \delta \right\rbrace \right)
\)
where $cl(\cdot)$ denotes the closure of the set in the complex plane.
\end{definition}

Clearly, $W_{\delta}(A)$ is a closed subset of the closure of the usual numerical
range. Following a small modification of Dekker’s theorem \cite{Dekker1969}, it can be readily verified that  $ W_{\delta}(A) $ is connected. The question of whether $ W_{\delta}(A) $ is convex is of interest; it is known to be convex when $ A $ is normal or when the Hilbert space is two-dimensional.

\begin{definition}
\label{def3}
Let $A \in \mathfrak{B}(\mathcal{H})$ and $0 \leq \delta \leq \|A\|$. The
$\delta$-numerical radius and the $ \delta$-norm of $ A $ are defined respectively as
\begin{align*}
\omega_{\delta}(A)  &= \sup\{|\lambda| : \lambda \in W_{\delta}(A)\},
\\
\Vert A \Vert_{\delta} &=\sup\lbrace \Vert A x \Vert:x\in S_{1}(\mathcal{H}), \ \Vert A x \Vert \geq \delta \rbrace .
\end{align*}
\end{definition}

An important early generalization was introduced for a single Hilbert space operator through the notion of the $\delta$-numerical radius, first studied by Cho and Takaguchi \cite{Takaguchi1981}, building on ideas originating in the work of Stampfli \cite{Stampfli1970}.

The joint $\delta$-numerical radius for an $n$-tuple of operators is defined as follows.
\begin{definition}
\label{def4}
Let $\mathbb{A} = (A_{1}, \ldots, A_{n}) \in \mathfrak{B}^{n}(\mathcal{H})$. The joint
$\delta$-numerical radius of $\mathbb{A}$ is defined as
\[
    \omega_{\delta}(\mathbb{A})
    = \sup \Biggl\{
        \Bigl( \sum\limits_{m=1}^{n} |\langle A_{m}x, x\rangle|^{2} \Bigr)^{\frac{1}{2}} :
        x \in S_{1}(\mathcal{H}),\ \|\mathbb{A} x\| \ge \delta
    \Biggr\},
    \quad 0 \le \delta \le \|\mathbb{A}\|.
\]
\end{definition}

It must be noted that there are different calculations of the numerical radius in terms of the norm of the given operator, the norm of the adjoint operator, the norm of the modulus of the considered operator, or in different combinations of the given operator, etc.  The detailed information on these results can be found in \cite{AbuOmar2019,Bourhim2017,Halmos1982,Moradi2021,Moradi2021b,Moslehian2017,Omidvar2020,Omidvar2021,Sababheh2019,SababhehMoradi2021,Sheikhhosseini2022,Shebrawi2023,Yamazaki2007}.

The concept of the numerical radius has been generalized in many different directions. It will now be given a new definition.
\begin{definition} [\!\!\cite{Alomari2024}]
\label{def5}
For a continuous strictly increasing surjective function $ f:[0,\infty)\rightarrow [0,\infty) $ and $n$-tuple of operators $\mathbb{A} = (A_{1}, \ldots, A_{n}) \in \mathfrak{B}^{n}(\mathcal{H})$\textnormal{:}
\begin{enumerate}[label=\upshape{\arabic*.}, ref=\upshape{\arabic*}, labelindent=5pt, leftmargin=*]
    \item $ \Vert \mathbb{A}x \Vert_{f}= f^{-1}\left( \sum\limits_{m=1}^{n}f\left( \Vert A_{m}x \Vert \right)  \right) , \ x\in \mathcal{H}; $
    \item  $ \Vert \mathbb{A} \Vert_{f}= \sup \left\lbrace \Vert \mathbb{A}x \Vert_{f}: x\in S_{1}(\mathcal{H})  \right\rbrace ; $
    \item $ \vert \langle \mathbb{A}x, x \rangle   \vert_{f} = f^{-1} \left( \sum\limits_{m=1}^{n}f\left( \vert \langle A_{m}x,x\rangle \vert \right)  \right), \ x\in \mathcal{H};  $
    \item
    $
    \omega_{f}(\mathbb{A})=\sup\left\lbrace \vert \langle \mathbb{A} x,x \rangle \vert_{f} : x\in S_{1} (\mathcal{H}) \right\rbrace
    $
which is called $f$-numerical radius of $ \mathbb{A}\in \mathfrak{B}^{n}(\mathcal{H})$.
    \end{enumerate}
\end{definition}

A new definition of the numerical radius is now given, on which all subsequent investigations in this paper will focus.

\begin{definition}
\label{def6}
For a continuous strictly increasing surjective function $ f:[0,\infty)\rightarrow [0,\infty) $ and $n$-tuple operator
$\mathbb{A} = (A_{1}, \ldots, A_{n}) \in \mathfrak{B}^{n}(\mathcal{H}), \ n\geq 1$:
\begin{enumerate}[label=\upshape{\arabic*.}, ref=\upshape{\arabic*}, labelindent=5pt, leftmargin=*]
\item For any $ 0\leq \delta \leq \Vert \mathbb{A} \Vert_{f}  $,
\(
\Delta_{(f,\delta)}(\mathbb{A})= \left\lbrace x\in S_{1}(\mathcal{H}) : \Vert \mathbb{A}x \Vert_{f}\geq \delta  \right\rbrace ;
\)
\item For any $ 0\leq \delta \leq \Vert \mathbb{A} \Vert_{f}  $,
\(
\omega_{(f,\delta )}(\mathbb{A})= \sup\left\lbrace \vert \langle \mathbb{A}x,x \rangle \vert_{f} : x\in \Delta_{(f,\delta)}(\mathbb{A})\right\rbrace ,
\)
which will be called joint $(f,\delta)$-numerical radius of the $n$-tuple of operators $ \mathbb{A};$
\item $ \Vert \mathbb{A} \Vert_{(f,\delta)}= \sup\left\lbrace \Vert \mathbb{A}x \Vert_{f}  : x\in \Delta_{(f,\delta)}(\mathbb{A})\right\rbrace . $
\end{enumerate}
\end{definition}

Recently, attention has been focused on the numerical radius concepts associated with operator tuples. In particular, the function numerical radius of an $n$-tuple of bounded linear operators $\mathbb{A}=(A_{1},\dots,A_{n})\in\mathfrak{B}^{n}(\mathcal{H})$ has been introduced and studied systematically. This notion unifies several well-known quantities in the operator theory. Indeed, in the special case $n=1$ and $\delta=0$, the function numerical radius reduces to the classical numerical radius.

Furthermore, when the function is chosen as $f(t)=t^{2}$ for $t\in[0,\infty)$ and $\delta=0$, $ \omega_{(t^{2},0)} $ coincides with the classical Euclidean operator norm of an operator $n$-tuple. This observation has been explored and exploited in various recent studies; see, for example, \cite{Drnovsek2014, Feki2024, FekiYamazaki2021, Gau2021, Guesba2024, Guesba2025,  Jana2024, Mal2023, Muller2020, WangWang2025} and the references therein.

It should also be noted that, in the case $\delta=0$, the function numerical radius was previously considered in \cite{Alomari2024}, where several generalizations of known inequalities and structural results were obtained.

Moreover, for $-\infty<\delta\leq 0$, the function numerical radius coincides with the case $\delta=0$, that is, $  \omega_{(f,\delta)}(\mathbb{A}) = \omega_{(f,0)}(\mathbb{A}) $  ensuring the consistency of the theory for nonpositive values of the parameter.

For any $ \delta\leq \Vert \mathbb{A} \Vert_{f}  $ and $ \mathbb{A}\in \mathfrak{B}^{n}(\mathcal{H}) $ it is clear that (see \cite[Thm. 2.1]{Alomari2024})
\[
\omega_{(f,\delta)}(\mathbb{A})\leq \omega_{(f,0)}(\mathbb{A})\leq \sum\limits_{m=1}^{n}\omega (A_{m}).
\]

On the other hand, it is easy to see that for any $ \tau \leq \delta \leq \Vert \mathbb{A} \Vert_{f} $ is true
\[
\omega_{(f,\delta)}(\mathbb{A})\leq \omega_{(f,\tau)}(\mathbb{A})\leq \omega_{(f,0)}(\mathbb{A})\leq \Vert \mathbb{A} \Vert_{f}.
\]

For clarification, we provide the following simple example.
\begin{example}
In the Eucleadean space $ \mathbb{R}^{2}(\mathbb{R}) $ consider the operator ($ 2$-tuple)
$ \mathbb{A}=(A_{1},A_{2}), \ \mathbb{A} \in \mathfrak{B}^{2}(\mathbb{R}^{2}), $
where
\[
A_{1}=\begin{pmatrix}
        0 & 1\\
        0 & 0
    \end{pmatrix} : \mathbb{R}^{2}\rightarrow \mathbb{R}^{2} \ \text{and} \
A_{2}=\begin{pmatrix}
        0 & 0\\
        1 & 0
    \end{pmatrix} : \mathbb{R}^{2}\rightarrow \mathbb{R}^{2}.
\]
Let $ f:[0,\infty )\rightarrow [0,\infty ), \ f(t)=t, \ 0\leq t <\infty$.
Then, for $ x=(x_{1},x_{2})\in \mathbb{R}^{2} $,
\begin{align*}
&\Vert A_{1}x \Vert = \vert x_{2} \vert , \  \Vert A_{2}x \Vert = \vert x_{1} \vert, \  \vert \langle A_{1}x,x \rangle \vert = \vert \langle A_{2}x,x \rangle \vert = \vert x_{1}x_{2} \vert,\\
&\Vert \mathbb{A} x \Vert_{f}= \Vert A_{1}x \Vert + \Vert A_{2}x \Vert = \vert x_{1} \vert + \vert x_{2} \vert ,\\
&
\vert \langle \mathbb{A} x,x  \rangle \vert_{f} = \vert \langle A_{1}x,x \rangle \vert + \vert \langle A_{2}x,x \rangle \vert =2 \vert x_{1}x_{2} \vert .
\\
& \Vert \mathbb{A} \Vert_{f} = \sup\left\lbrace \Vert \mathbb{A}x \Vert_{f}: x \in S_{1}(\mathbb{R}^{2}) \right\rbrace  \\
&\quad = \sup \left\lbrace \vert x_{1} \vert + \vert x_{2} \vert : x=(x_{1},x_{2})\in S_{1}(\mathbb{R}^{2}) \right\rbrace  \\
&\quad = \sup \left\lbrace \sqrt{1+2\vert x_{1}x_{2} \vert}: (x_{1},x_{2})\in S_{1}(\mathbb{R}^{2}) \right\rbrace = \sqrt{2}.
\end{align*}
Then, for any $ 0\leq \delta \leq \sqrt{2} $ and $ (x_{1},x_{2}) \in S_{1}(\mathbb{R}^{2})$ from the inequality $ \Vert \mathbb{A}x \Vert_{f}\geq \delta $ we have $ \vert x_{1} \vert + \vert x_{2} \vert \geq \delta . $ Hence $ 1+2\vert x_{1}x_{2} \vert  \geq \delta^{2} $ and from this $ \vert x_{1}x_{2} \vert \geq \dfrac{\delta^{2}-1}{2}, \ \delta\leq \sqrt{2}. $
Therefore, for $ 0\leq \delta \leq \sqrt{2} $,
\begin{multline*}
\Delta_{(f,\delta)}(\mathbb{A})  =  \left\lbrace x=(x_{1},x_{2})\in S_{1}(\mathbb{R}^{2}) : \Vert \mathbb{A} x \Vert_{f}\geq \delta  \right\rbrace \\
 =  \left\lbrace (x_{1},x_{2})\in S_{1}(\mathbb{R}^{2}) : \vert x_{1}x_{2}\vert \geq \dfrac{\delta^{2}-1}{2}   \right\rbrace \nonumber
\end{multline*}
Consequently, we have
\begin{eqnarray}
\omega_{f,\delta}(\mathbb{A})  =  \left\lbrace \vert  \langle  \mathbb{A} x,x\rangle\vert_{f}: x\in \Delta_{(f,\delta)}(\mathbb{A}) \right\rbrace
 =  \begin{cases}
1, & \text{if} \ 0 \leq \delta \leq 1, \\
\delta^{2}-1, &  \text{if} \ 1 < \delta <\sqrt{2}.  \nonumber
\end{cases}
\end{eqnarray}
\end{example}
\begin{definition}
\label{def7}
For any fixed $ 0\leq \alpha \leq \dfrac{\pi}{2}, $ let \[ S(\alpha):=\left\lbrace re^{i\varphi}: r\geq 0, \ -\alpha < \varphi <\alpha \right\rbrace \subset \mathbb{C}. \]
\begin{enumerate}[label=\upshape{\arabic*.}, ref=\upshape{\arabic*}, labelindent=5pt, leftmargin=*]
\item The operator $ A:\mathcal{H}\rightarrow \mathcal{H}, \ A\in \mathfrak{B}(\mathcal{H})  $ in any Hilbert space $ \mathcal{H} $ is said to be $ \alpha$-sectorial, if $ W(A)\subset S(\alpha); $
\item The class of all $ \alpha$-sectorial operators on Hilbert space $ \mathcal{H} $ with vertex in origin and half-angle $ \left[ 0,\frac{\pi}{2}\right)  $ will be denoted by $ Sec_{\alpha}(\mathcal{H}).$
\end{enumerate}
\end{definition}
It is clear that if $ A\in Sec_{\alpha}(\mathcal{H}), \ 0\leq \alpha \leq \dfrac{\pi}{2},$ then $ A $ is accretive.
Recall that an operator $ A\in \mathfrak{B}(\mathcal{H}) $
is called accretive if $ Re \langle Ax, x \rangle \geq 0  $ for all $ x\in\mathcal{H}.$
It is called dissipative if $ Im \langle Ax, x \rangle \geq 0  $ for all $ x\in\mathcal{H}$.

The main goal of this study is to generalize well-known results concerning the joint $ f$-numerical radius for $n$-tuple Hilbert space operators, as found in the existing mathematical literature, to the framework of the joint $(f,\delta)$-numerical radius. In Section \ref{Sec:2}, this work aims to provide new and diverse contributions to this area. In Section \ref{Sec:3}, novel and interesting estimates for the lower and upper bounds of the $(f,\delta)$-numerical radius function are established in the context of sectorial operators.

The results obtained in this study not only generalize and extend many well-known inequalities associated with the classical and joint numerical radius functions, but also improve several of them. Moreover, this study provides strengthened formulations and new perspectives that significantly enrich existing literature.

Throughout this article, $ f:[0,\infty)\rightarrow [0,\infty) $ denotes a continuous, strictly increasing, and surjective function.  Note that under these conditions, $ f(0)=0$. Furthermore, if $ f $ is a convex function in $ [0,\infty ), $ then for any $ 0\leq \alpha \leq 1$ and $0\leq x < \infty$, we have $ f(\alpha x )\leq \alpha f(x)$. For $ \mathbb{A}\in \mathfrak{B}^{n}(\mathcal{H})$ and $\lambda\in \mathbb{R} $, we define
\(
\Delta_{(f,\delta)}(\mathbb{A}) = \left\lbrace x\in S_{1}(\mathcal{H}): \Vert \mathbb{A}x \Vert_{f}\geq\lambda\right\rbrace.
\)
A function $f : [0,\infty) \to [0,\infty)$ is said to be \emph{geometrically convex} if
\(
f(\sqrt{xy}) \le \sqrt{f(x)f(y)}
\)
holds for all $x,y \in [0,\infty)$.

\section{Joint \texorpdfstring{$(f,\delta)$}{(f,delta)}-numerical radius of \texorpdfstring{$n$}{n}-tuples of operators}
\label{Sec:2}

In this section, we first investigate some important  properties of the joint $(f,\delta)$-numerical radius function of bounded linear $n$-tuples of operators on Hilbert spaces. Subsequently, some generalization formulas
for the lower and upper bounds of the joint $(f,\delta)$-numerical radius of bounded linear $n$-tuple of operators on Hilbert spaces are obtained.

\begin{theorem} \label{thm-fdnumradntupop}
\label{thm1}
Let $\mathbb{A} = (A_{1}, \ldots, A_{n}), \ \mathbb{B} = (B_{1}, \ldots, B_{n}) \in \mathfrak{B}^{n}(\mathcal{H})$.
\begin{enumerate}[label=\upshape{\arabic*.}, ref=\upshape{\arabic*}, labelindent=5pt, leftmargin=*,ref=\upshape{\arabic*}]
\item \label{item1-thm-fdnumradntupop} For all $ \lambda \in \mathbb{C}, \ \lambda\neq 0 $ and $ 0\leq \delta \leq \Vert \lambda \mathbb{A} \Vert_{f}, $
\[
\omega_{(f,\delta)}(\lambda \mathbb{A} )= \vert \lambda \vert \omega_{\left(f, \frac{\delta}{\vert \lambda \vert}\right) } ( \mathbb{A} ),
\]
provided that $ f $ satisfies the equality $ f(xy)=f(x)f(y) $ for $ x,y\geq 0. $

\item \label{item2-thm-fdnumradntupop} If $ f $ is geometrically convex, then for any $ 0\leq \delta \leq \Vert \mathbb{A} + \mathbb{B} \Vert_{f} ,$
\[
\omega_{(f,\delta)}(\mathbb{A} + \mathbb{B} )\leq \omega_{\left( f, \delta- \Vert \mathbb{B} \Vert_{f}\right) } (\mathbb{A})+ \omega_{\left( f, \delta- \Vert \mathbb{A} \Vert_{f}\right) }(\mathbb{B}) .
\]
In special case, for any $ \alpha \leq \max \left\lbrace \Vert \mathbb{A} \Vert_{f}, \Vert \mathbb{B} \Vert_{f} \right\rbrace, $
\[
\omega_{(f,\delta)}(\mathbb{A} + \mathbb{B} )\leq \omega_{\left( f, \delta-\alpha\right) }(\mathbb{A}) + \omega_{\left( f, \delta-\alpha\right) } (\mathbb{B}).
\]
\item \label{item3-thm-fdnumradntupop} If $\mathbb{A}$ is $n$-tuple of normal operators in $ \mathfrak{B}^{n}(\mathcal{H}),$ then for any $ 0\leq \delta \leq  \Vert \mathbb{A} \Vert_{f} $,
\[
\omega_{(f,\delta)}(\mathbb{A}) = \omega_{(f,\delta)}(\mathbb{A}^{*}).
\]
\item \label{item4-thm-fdnumradntupop} If $\mathbb{A}$ is $n$-tuple of hyponormal operators in $ \mathfrak{B}^{n}(\mathcal{H}) ,$ then for any $ 0\leq \delta \leq  \Vert \mathbb{A}^{*} \Vert_{f}$,
\[
\omega_{(f,\delta)}(\mathbb{A}^{*}) \leq \omega_{(f,\delta)}(\mathbb{A}).
\]
\end{enumerate}
\end{theorem}
\begin{proof}
\noindent \ref{item1-thm-fdnumradntupop}.\ For $ \lambda\in \mathbb{C}, \ \lambda\neq 0 $ and $ x\in \mathcal{H}, $ since $ f $ satisfies $ f(xy)=f(x)f(y) $, then
\begin{align}
& \vert \langle \lambda \mathbb{A}x,x \rangle\vert_{f}  =  f^{-1}  \left( \sum\limits_{m=1}^{n}f\left( \vert \langle\lambda A_{m}x,x\rangle \vert \right)  \right) \nonumber \\
& \quad =  f^{-1}  \left( \sum\limits_{m=1}^{n}f\left( \vert \lambda \vert\vert \langle A_{m}x,x\rangle \vert \right)  \right)
  =  f^{-1}  \left( \sum\limits_{m=1}^{n}f\left( \vert \lambda \vert\right)f\left(  \vert \langle A_{m}x,x\rangle \vert \right)  \right) \nonumber\\
& \quad =  \vert \lambda \vert f^{-1}  \left( \sum\limits_{m=1}^{n}f\left( \vert \langle A_{m}x,x\rangle \vert \right)  \right)
  =  \vert \lambda \vert \vert \langle \mathbb{A}x,x \rangle\vert_{f}. \label{equ3}
\end{align}
Since for any $ x\in \Delta_{(f,\delta )}(\lambda \mathbb{A} ), \ \lambda \in \mathbb{C}\setminus \lbrace 0 \rbrace  $ the following equality holds
\(
\Vert \lambda  \mathbb{A}x  \Vert_{f} = \vert \lambda \vert \Vert \mathbb{A}x \Vert_{f},
\)
it follows that
\begin{equation}
\label{equ4}
\Delta_{(f,\delta)}(\lambda \mathbb{A})=\Delta_{\left(f, \frac{\delta}{\vert \lambda \vert}\right) } ( \mathbb{A} ).
\end{equation}
Therefore, \eqref{equ3} and \eqref{equ4} imply
\(
\omega_{(f,\delta)}(\lambda \mathbb{A} )= \vert \lambda \vert \omega_{\left(f, \frac{\delta}{\vert \lambda \vert}\right) } ( \mathbb{A} )
\)
for $ 0\leq \delta \leq \Vert \lambda \mathbb{A} \Vert_{f}. $

\noindent \ref{item2-thm-fdnumradntupop}.\ Since $ f $ monotonically increasing,
\begin{multline*}
\sum\limits_{m=1}^{n}f\left( \vert \langle (A_{m}+B_{m})x,x\rangle \vert \right)  =  \sum\limits_{m=1}^{n}f\left( \vert \langle A_{m}x,x\rangle + \langle B_{m}x,x\rangle \vert \right)  \\
\leq  \sum\limits_{m=1}^{n}f\left( \vert \langle A_{m}x,x\rangle\vert + \vert\langle B_{m}x,x\rangle \vert \right).
\end{multline*}
Also, from \cite{Mulholland1950} we obtain
\begin{multline*}
f^{-1} \left(  \sum\limits_{m=1}^{n}f\left( \vert \langle A_{m}x,x\rangle\vert + \vert\langle B_{m}x,x\rangle \vert \right) \right) \leq f^{-1} \left(  \sum\limits_{m=1}^{n}f\left( \vert \langle A_{m}x,x\rangle\vert \right) \right) \\
+ f^{-1} \left(  \sum\limits_{m=1}^{n}f\left( \vert \langle B_{m}x,x\rangle\vert \right) \right).
\end{multline*}
Consequently, one can easily see that
\begin{eqnarray}
\label{equ5}
\begin{array}{r}
f^{-1} \left( \sum\limits_{m=1}^{n}f\left( \vert \langle (A_{m}+B_{m})x,x\rangle \vert \right)  \right) \leq f^{-1} \left(  \sum\limits_{m=1}^{n}f\left( \vert \langle A_{m}x,x\rangle\vert \right) \right) \\
+ f^{-1} \left(  \sum\limits_{m=1}^{n}f\left( \vert \langle B_{m}x,x\rangle\vert \right) \right).
\end{array}
\end{eqnarray}
On the other hand, it is clear that for any $x \in S_{1}(\mathcal{H})$,
\[
\sum\limits_{m=1}^n f\left( \Vert (A_m + B_m)x \Vert \right) \leq
\sum\limits_{m=1}^n f\left(  \Vert A_m x\Vert + \Vert B_m x\Vert \right) .
\]
Hence, by \cite[Corollary 1.1]{Mulholland1950}, one obtains that
\begin{multline*}
f^{-1}\left(
\sum\limits_{m=1}^n f\left( \Vert(A_m + B_m)x\Vert\right)
\right)
\leq
f^{-1}\left(
\sum\limits_{m=1}^n f\left(  \Vert A_m x\Vert \right)
\right) \\
 \quad +
f^{-1}\left(
\sum\limits_{m=1}^n f\left( \Vert B_m x\Vert\right)
\right) \leq
f^{-1}\left(
\sum\limits_{m=1}^n f\left(  \Vert A_m x\Vert \right)
\right)
+
\Vert \mathbb{B} \Vert_{f},
\end{multline*}
that is,
\(
\Vert (\mathbb{A}+\mathbb{B})x\Vert_f
\leq
\Vert \mathbb{A}x\Vert_f + \Vert \mathbb{B}\Vert_f.
\)
From the last inequality, for $x \in \Delta_{\delta}({\mathbb{A}+\mathbb{B}})$, one can easily see that
\[
\delta \leq \Vert (\mathbb{A}+\mathbb{B})x\Vert_f \leq \Vert \mathbb{A}x\Vert_f + \Vert \mathbb{B}\Vert_f.
\]
that is,
\(
\delta - \Vert \mathbb{B} \Vert_{f} \leq \Vert (\mathbb{A}+\mathbb{B})x\Vert_{f}.
\)
Similarly, it can be shown that
\(
\delta - \Vert \mathbb{A} \Vert_{f} \leq \Vert (\mathbb{A}+\mathbb{B})x \Vert_{f}
\)
for $x \in \Delta_{\delta}({\mathbb{A}+\mathbb{B}}).$
Thus,
\begin{equation}
\label{equ6}
\Delta_{(f,\delta)}(\mathbb{A}+\mathbb{B})
\subset
\Delta_{(f,\delta-\Vert \mathbb{B} \Vert_{f})}(\mathbb{A})
\cap
\Delta_{(f,\delta-\Vert \mathbb{A} \Vert_{f})}(\mathbb{B}).
\end{equation}
Consequently, from the relations \eqref{equ5} and \eqref{equ6}, it follows that
\[
\omega_{(f,\delta)}(\mathbb{A}+\mathbb{B})
\le
\omega_{(f,\delta-\Vert \mathbb{B} \Vert_{f})}(\mathbb{A})
+
\omega_{(f,\delta-\Vert \mathbb{A} \Vert_{f})}(\mathbb{B}),
\]
for any $ 0 \leq \delta \leq \Vert \mathbb{A} + \mathbb{B} \Vert_{f}$.

\noindent \ref{item3-thm-fdnumradntupop}.\ If $\mathbb{A}  \in \mathfrak{B}^{n}(\mathcal{H})$ is an $n$-tuple of operators, then for $x \in \mathcal{H}$,
\begin{equation}
\label{equ7}
\begin{array}{l}
\vert\langle \mathbb{A}x, x \rangle\vert_{f}
= f^{-1}\left( \sum\limits_{m=1}^{n} f\left(  \vert\langle A_m x, x \rangle\vert \right)  \right) \\
= f^{-1}\left( \sum\limits_{m=1}^{n} f\left(  \vert\langle A_m^{*} x, x \rangle\vert \right)  \right)
= \vert\langle \mathbb{A}^{*} x, x \rangle\vert_{f}.
\end{array}
\end{equation}
On the other hand, since $ A_{m}, \ 1\leq m \leq n $ is normal operator in $ \mathcal{H}, $ then
for $0 \leq \delta \leq \Vert \mathbb{A} \Vert_{f}$, it is clear that
\[
\delta \leq \Vert \mathbb{A} x \Vert_{f}
= f^{-1}\left( \sum\limits_{m=1}^{n} f\left(  \Vert A_m x \Vert \right)  \right)
= f^{-1}\left( \sum\limits_{m=1}^{n} f\left(  \Vert A_m^{*} x\Vert \right)  \right)
= \Vert A x\Vert_{f}.
\]
Then, for $0 \leq \delta \leq \Vert \mathbb{A} \Vert_{f}$, it is true that
\begin{equation}
\label{equ8}
\Delta_{(f,\delta)}(\mathbb{A})
=
\Delta_{(f,\delta)}(\mathbb{A}^{*}).
\end{equation}
Therefore, from \eqref{equ7} and \eqref{equ8}, it follows that
\(
\omega_{(f,\delta)}(\mathbb{A})
=
\omega_{(f,\delta)}(\mathbb{A}^{*}).
\)

\noindent \ref{item4-thm-fdnumradntupop}.\ Since $ \mathbb{A} $ is hyponormal operator, then for
$ 0 \leq \delta \leq \Vert A^{*}\Vert_{f}, $
\[
\delta \leq \Vert \mathbb{A}^{*}x\Vert_{f} \leq \Vert \mathbb{A}x\Vert_{f}, \
x \in S_{1}(\mathcal{H}),
\]
that is,
\begin{equation}
\label{equ9}
\Delta_{(f,\delta)}(\mathbb{A}^{*})
\subset
\Delta_{(f,\delta)}(\mathbb{A}).
\end{equation}
Thus, \eqref{equ7} and \eqref{equ9} yield
\( \omega_{(f,\delta)}(\mathbb{A}^{*}) \leq \omega_{(f,\delta)}(\mathbb{A}).   \)
\end{proof}

If we take $f(t)=t^{2}$, $0 \leq t < \infty $, and $\delta = 0$ in the statement \ref{item2-thm-fdnumradntupop} of Theorem \ref{thm1}, we obtain the following corollary, which was proved in \cite{Jana2024}.

\begin{corollary}
\label{cor1}
Let $\mathbb{A}, \ \mathbb{B} \in \mathfrak{B}^{n}(\mathcal{H})$. Then,
\[
\omega_{(t^{2},0)}(\mathbb{A}+\mathbb{B})
\leq
\omega_{(t^{2},0)}(\mathbb{A})
+
\omega_{(t^{2},0)}(\mathbb{B}).
\]
\end{corollary}

If we take $n=1$ in Corollary \ref{cor1}, then Corollary \ref{cor1} coincides with the following well-known result.
\begin{corollary}
\label{cor2}
Let $A_{1}, \ B_{1} \in \mathfrak{B}(\mathcal{H})$. Then,
\[
\omega(A_{1}+B_{1})
\leq
\omega(A_{1})
+
\omega(B_{1}).
\]
\end{corollary}
\begin{remark}
Theorem \ref{thm1} generalizes \textnormal{\cite[Proposition 2.2]{Ismailov2026}}.
\end{remark}

\begin{theorem}
\label{thm2}
Let $ f $ be geometrically convex function,  $
\mathbb{A}\in \mathfrak{B}^{n}(\mathcal{H}) $ be $n$-tuple of operators and $ \ 0\leq \delta \leq \Vert \mathbb{A} \Vert_{f}.
$ Then
\[
f^{2}\left( \omega_{(f,\delta)}(\mathbb{A} )\right)
\leq
\sqrt{n}f(1) f\left(
\omega_{\left( f, \frac{f^{2}(\delta )}{nf(1)} \right) }(\mathbb{A}^{*}\mathbb{A}) \right) .
\]
\end{theorem}

\begin{proof}
Since $ f $ is geometrically convex function, then for any $x\in S_{1}(\mathcal{H})$,
\begin{multline}
\label{equ10}
f\left( \vert\langle \mathbb{A} x,x\rangle \vert_{f} \right)
=
\sum\limits_{m=1}^n
f\left( \vert \langle A_m x,x\rangle \vert \right) \leq
\sum\limits_{m=1}^n
f\left( \Vert A_m x\Vert \right)  \\
=
\sum\limits_{m=1}^n
f\left( \langle A_m^*A_m x,x\rangle^{1/2} \right) \leq \sqrt{f(1)}
\sum\limits_{m=1}^n
 f^{1/2}\left( \langle A_m^*A_m x,x\rangle\right)  \\
\leq
\sqrt{nf(1)}\left( \sum\limits_{m=1}^n
f\left(  \langle A_m^*A_m x,x\rangle\right)  \right) ^{1/2} =
\sqrt{nf(1)}\left(
f\left(  \langle A^{*}A x,x\rangle\right)  \right) ^{1/2}.
\end{multline}
Since $ f $ is geometrically convex function, for any $ x \in \Delta_{(f,\delta)}(\mathbb{A}), \ 0 \leq \delta \leq \Vert \mathbb{A} \Vert_f, $ it holds
\begin{align*}
f(\delta)
&\leq f(\Vert\mathbb{A} x\Vert_f) = \sum\limits_{m=1}^n f(\Vert A_m x\Vert) \\
&= \sum\limits_{m=1}^n f\left( \langle A_m^* A_m x, x\rangle^{1/2} \right) \leq \sum\limits_{m=1}^n f\left( \Vert A_m^* A_m x\Vert^{1/2} \right)  \\
&\leq \sqrt{f(1)}\sum\limits_{m=1}^n  f^{1/2} \left( \Vert A_m^* A_m x\Vert\right) \leq \sqrt{nf(1)}\left( \sum\limits_{m=1}^n  f \left( \Vert A_m^* A_m x\Vert\right)  \right)^{1/2}  \\
&= \sqrt{n f(1)}f^{1/2}(\Vert \mathbb{A}^* \mathbb{A} x\Vert_f) .
\end{align*}
Hence,
\(
\dfrac{f^2(\delta)}{n f(1)} \leq f(\Vert \mathbb{A}^* \mathbb{A} x\Vert_f),
\)
that is,
\begin{equation}
\label{equ11}
\Delta_{(f,\delta)}(\mathbb{A})
\subset
\Delta_{\left( f,\frac{f^2(\delta)}{n f(1)}\right) }(\mathbb{A}^*\mathbb{A}).
\end{equation}
Therefore, from \eqref{equ10} and \eqref{equ11} it follows that
\[
f^{2}\left( \omega_{(f,\delta)}(\mathbb{A} )\right)
\leq
\sqrt{n}f(1) f\left(
\omega_{\left( f, \frac{f^{2}(\delta )}{nf(1)} \right) }(\mathbb{A}^{*}\mathbb{A}) \right) . \tag*{\qedhere}
\]
\end{proof}

We now provide several useful auxiliary lemmas.
\begin{lemma}
\label{lem1}
Let $\mathbb{A}\in \mathfrak{B}^{n}(\mathcal{H})$ be $n$-tuple of operators. Then,
\(
\Vert \mathbb{A}\Vert_{f} = \Vert \ \vert  \mathbb{A} \vert \  \Vert_f.
\)
\end{lemma}
\begin{proof} For any $ x\in \mathcal{H}, $ it is clear that
\begin{multline}
\label{equ12}
 f\left( \Vert \mathbb{A} x \Vert_{f} \right) = \sum\limits_{m=1}^{n}f\left( \Vert A_{m}x \Vert \right)  = \sum\limits_{m=1}^{n}f\left( \langle A_{m}x,A_{m}x\rangle^{1/2} \right)  \\
 = \sum\limits_{m=1}^{n}f\left( \langle A_{m}^{*}A_{m}x,x\rangle^{1/2} \right)=\sum\limits_{m=1}^{n}f\left( \langle \vert A_{m} \vert^{2}x,x \rangle ^{1/2}  \right)  \\
 =\sum\limits_{m=1}^{n}f\left( \Vert \ \vert A_{m} \vert x  \Vert \right) =f\left( \Vert \ \vert \mathbb{A} \vert x  \Vert_{f} \right).
\end{multline}
Taking supremum over all $ x\in S_{1}(\mathcal{H}) $ in \eqref{equ12} yields
\(\Vert \mathbb{A}\Vert_{f} = \Vert \ \vert  \mathbb{A} \vert \  \Vert_f.\)
\end{proof}

\begin{lemma}
\label{lem2}
Let $\mathbb{A}\in \mathfrak{B}^{n}(\mathcal{H})$ be $n$-tuple normal operator. Then,
\[
\Vert \mathbb{A}^{*}\Vert_{f} = \Vert \mathbb{A} \Vert_f.
\]
\end{lemma}
\begin{proof}  For any $ x\in \mathcal{H}, $ it is clear that
\begin{align}
\label{equ13}
f\left( \Vert \mathbb{A}^{*} x \Vert_{f} \right)
&= \sum\limits_{m=1}^{n}f\left( \Vert A_{m}^{*}x \Vert \right) = \sum\limits_{m=1}^{n}f\left( \langle A_{m}^{*}x,A_{m}^{*}x\rangle^{1/2} \right) \nonumber \\
&= \sum\limits_{m=1}^{n}f\left( \langle A_{m}A_{m}^{*}x,x\rangle^{1/2} \right) = \sum\limits_{m=1}^{n}f\left( \langle A_{m}^{*}A_{m}x,x\rangle^{1/2} \right) \nonumber \\
&=\sum\limits_{m=1}^{n}f\left( \langle \vert A_{m}^{*} \vert^{2}x,x \rangle ^{1/2}  \right) =\sum\limits_{m=1}^{n}f\left( \Vert \ \vert A_{m} \vert x  \Vert \right)
=f\left( \Vert  \mathbb{A} x  \Vert_{f} \right) .
\end{align}
Taking supremum over all $ x\in S_{1}(\mathcal{H})  $ in \eqref{equ13} yields
\(
\Vert \mathbb{A}^{*}\Vert_{f} = \Vert  \mathbb{A}  \Vert_f.
\)
\end{proof}
\begin{lemma}
\label{lem3}
Let $\mathbb{A}\in \mathfrak{B}^{n}(\mathcal{H})$ be $n$-tuple of operators and $ 0\leq \delta \leq \Vert \mathbb{A} \Vert_{f}. $
 Then
\[
\Vert \mathbb{A}\Vert_{(f,\delta)} = \Vert \mathbb{A} \Vert_f.
\]
\end{lemma}

\begin{proof}
It is clear that
\begin{align*}
\Vert \mathbb{A} \Vert_f
&= \sup \left\lbrace  \Vert \mathbb{A} x \Vert_f : x \in S_{1}(\mathcal{H}) \right\rbrace  \\
&= \max \left\lbrace
\sup \left\lbrace  \Vert \mathbb{A} x \Vert_f : \
x \in S_{1}(\mathcal{H}), \
\Vert \mathbb{A} x\Vert_f \geq \delta\right\rbrace, \right. \\
&\hspace{4cm} \left. \sup \left\lbrace  \Vert \mathbb{A} x \Vert_f: \
x \in S_{1}(\mathcal{H}), \ 0\leq
\Vert \mathbb{A} x\Vert_f < \delta\right\rbrace \right\rbrace \\
&\leq \max \left\lbrace
\sup \left\lbrace \Vert \mathbb{A} x \Vert_{f} : \
x \in S_{1}(\mathcal{H}), \
\Vert \mathbb{A} x\Vert_f \geq \delta\right\rbrace, \delta \right\rbrace \\
&= \max \left\lbrace \delta, \Vert \mathbb{A} \Vert_{(f,\delta )}  \right\rbrace .
\end{align*}
If $\Vert\mathbb{A}\Vert_{(f,\delta)} \geq \delta$, then the above inequality implies that
\(
\Vert \mathbb{A} \Vert_f \leq \Vert \mathbb{A} \Vert_{(f,\delta)}.
\)
Moreover, it is clear that
\(
\Vert \mathbb{A} \Vert_{(f,\delta)} \leq \Vert \mathbb{A} \Vert_f.
\)
Finally, we obtain
\(
\Vert \mathbb{A} \Vert_{(f,\delta)} = \Vert A \Vert_f.\)
\end{proof}

First, we present a well-known inequality that is essential for proving the theorems.

\begin{lemma} \textnormal{(\!\!\cite{Berezansky})}
\label{lem4}
If $ A\in \mathfrak{B}(\mathcal{H}) , \ A\geq 0 ,$ then for any $ x\in \mathbb{A} $,
\[
\Vert Ax \Vert^{2}\leq \Vert A \Vert \langle Ax,x\rangle .
\]
\end{lemma}

The next result establishes the relationship between the $(f,\delta)$-numerical radius of an $n$-tuple of operators and its modulus.
\begin{theorem}
\label{thm3}
Let $ f $ be geometrically convex function, $ \mathbb{A} \in \mathfrak{B}^{n}(\mathcal{H}) $ be $n$-tuple of operators and
$0 \leq \delta \leq \Vert \mathbb{A}\Vert_f.$ Then,
\(
f\left( \Vert \mathbb{A} \Vert_{f} \right)
\leq n
f\left( \omega_{(f,\delta )}( \vert \mathbb{A} \vert )\right)
.
\)
\end{theorem}
\begin{proof}
Since $f$ is geometrically convex, then
for any $x \in \mathcal{H}$ by Lemma \ref{lem4} it follows that
\begin{align*}
f(\Vert\mathbb{A}x\Vert_f)
&= \sum\limits_{m=1}^n f(\Vert A_m x \Vert) = \sum\limits_{m=1}^n f\left( \langle A_m^*A_m x, x \rangle^{1/2} \right) = \sum\limits_{m=1}^n f\left( \Vert \ \vert A_m \vert x \Vert \right) \\
&= \sum\limits_{m=1}^n f\left( \left( \Vert \ \vert A_m \vert x \Vert^{2} \right)^{1/2} \right)  \leq \sum\limits_{m=1}^n f\left( \Vert \ \vert A_m \vert \ \Vert^{1/2} \langle \vert A_{m} \vert x,x \rangle^{1/2}\right)  \\
&\leq \sum\limits_{m=1}^n \left( f\left( \Vert \ \vert A_m \vert \ \Vert\right)\right)  ^{1/2} \left( f \left( \langle \vert A_{m} \vert x,x \rangle\right)\right) ^{1/2} \\
&\leq \left( f\left( \Vert \ \vert \mathbb{A} \vert \ \Vert_{f} \right)  \right) ^{1/2} \sum\limits_{m=1}^{n} \left( f\left( \langle \vert A_{m} \vert x,x \rangle \right)  \right) ^{1/2} \\
&\leq \sqrt{n} \left( f  \left( \Vert \ \vert \mathbb{A} \vert \ \Vert_{f} \right)\right)^{1/2} \left( \sum\limits_{m=1}^n f \left( \langle \vert A_{m} \vert x,x \rangle \right) \right) ^{1/2} \\
&= \sqrt{n} f^{1/2} \left( \Vert \ \vert \mathbb{A} \vert \ \Vert_{f} \right)f^{1/2}\left( \langle \vert \mathbb{A} \vert x,x \rangle_{f} \right),
\end{align*}
that is,
\begin{equation}
\label{equ14}
f^{2}(\Vert\mathbb{A}x\Vert_f)\leq n f\left( \Vert \ \vert \mathbb{A} \ \vert \Vert_{f} \right)f\left( \langle \vert \mathbb{A} \vert x,x \rangle_{f} \right).
\end{equation}
Moreover, for any $ x\in S_{1}(\mathcal{H}) $ and $ 0\leq \delta \leq \Vert \mathbb{A} \Vert_{f} $, since
\begin{multline*}
f(\delta)\leq \sum\limits_{m=1}^n f\left( \Vert A_{m}x \Vert \right) = \sum\limits_{m=1}^n f \left( \langle A_{m}^*A_{m}x,x \rangle^{1/2} \right) \\
= \sum\limits_{m=1}^n f \left( \Vert \ \vert A_{m} \vert x \Vert  \right) = f \left( \Vert \ \vert \mathbb{A} \vert \ x \Vert  \right) ,
\end{multline*}
we obtain
\begin{equation}
\label{equ15}
\Delta_{(f,\delta )}(\mathbb{A}) = \Delta_{(f,\delta)}(\vert \mathbb{A} \vert).
\end{equation}
Therefore, \eqref{equ14} and \eqref{equ15} yields
\(
f^{2}\left( \Vert \mathbb{A} \Vert_{(f,\delta )} \right)
\leq n f \left( \Vert \ \vert \mathbb{A}  \vert \ \Vert_{f}   \right)
f\left( \omega_{(f,\delta )}( \vert \mathbb{A} \vert )\right)
.
\)
By Lemma \ref{lem3} we have
\[
f^{2}\left( \Vert \mathbb{A} \Vert_{f} \right)
\leq n f \left( \Vert \ \vert \mathbb{A} \vert \ \Vert_{f}   \right)
f\left( \omega_{(f,\delta )}( \vert \mathbb{A} \vert )\right)
.
\]
By Lemma \ref{lem1} it follows that
\(
f^{2}\left( \Vert \mathbb{A} \Vert_{f} \right)
\leq n f \left( \Vert \mathbb{A} \Vert_{f}  \right)
f\left( \omega_{(f,\delta )}( \vert \mathbb{A} \vert )\right)
.
\)
Thus, we have
\(
f\left( \Vert \mathbb{A} \Vert_{f} \right)
\leq n
f\left( \omega_{(f,\delta )}( \vert \mathbb{A} \vert )\right)
.
\)
\end{proof}

If we take $ f(t)=t^{2}, \ 0\leq t< \infty  $ in Theorem \ref{thm3}, then we obtain the following corollaries.
\begin{corollary}
\label{cor3}
Let $ \mathbb{A} \in \mathfrak{B}^{n}(\mathcal{H}) $ be $n$-tuple of operators and
$0 \leq \delta \leq \Vert \mathbb{A}\Vert_{t^{2}}.$  Then,
\(
\Vert \mathbb{A} \Vert_{t^{2}} \leq \sqrt{n} \omega_{(t^{2},\delta)}(\vert \mathbb{A} \vert ).
\)
\end{corollary}

\begin{corollary}
\label{cor4}
Let $\mathbb{A} = (A_1,\dots,A_n) \in \mathfrak{B}^n(\mathcal{H})$, $ A_{m}\geq 0, \ m=1,\ldots, n $ and $ 0\leq \delta \leq \Vert \mathbb{A} \Vert_{t^{2}}$. Then,
\(
\omega_{(t^{2},\delta)}(\mathbb{A})\geq \frac{\Vert \mathbb{A} \Vert_{t^{2}}}{\sqrt{n}}.
\)
\end{corollary}

If we take $ \delta =0 $ in  Corollary \ref{cor4}, then the obtained result is an improved version of the result obtained in \cite{Drnovsek2014}.

Now, it will be presented two theorems that describes the relationship between
the $(f,\delta)$-numerical radius of an $n$-tuple of operators and the $\delta$-numerical
radius of its coordinate operators.
\begin{theorem}
\label{thm4}
Let $  \mathbb{A} = (A_1, \ldots, A_n) \in \mathfrak{B}^{n}(\mathcal{H})$  be $n$-tuple of operators and $
0 \leq \delta_m \leq \Vert A_m \Vert,\; m=1,\ldots,n. $ Then
\[
\frac{1}{n} \sum\limits_{m=1}^n \omega_{\delta_m}(A_m)
\leq
\omega_{(f,\delta)}(\mathbb{A}),
\]
where
\(
\delta = f^{-1}\left( \sum\limits_{m=1}^n f(\delta_m) \right).
\)

\end{theorem}
\begin{proof}
For any $x \in \Delta_{\delta_m}(A_m)$, $m=1,\ldots,n$,
\[
0 \leq f(\delta_m)
\leq \sum\limits_{m=1}^n f(\delta_m)
\leq \sum\limits_{m=1}^n f(\Vert A_m x\Vert)
= f(\Vert \mathbb{A}x\Vert_f).
\]
Hence, for any $m = 1,\ldots,n$,
\(
\Delta_{\delta_m}(A_m) \subset \Delta_{(f,\delta)}(\mathbb{A}),
\
\delta = f^{-1}\left( \sum\limits_{m=1}^n f(\delta_m) \right).
\)
Thus, since
\(
f\left(\vert \langle A_{m}x,x\rangle \vert \right)  \leq \sum\limits_{m=1}^n f\left( \vert \langle A_{m}x,x\rangle \vert \right)
= f\left( \vert \langle \mathbb{A}x,x\rangle  \vert_{f} \right) , \ x \in \mathcal{H}
\),
it follows that
\(
\omega_{\delta_m}(A_m)
\leq
\omega_{(f,\delta)}(\mathbb{A}), \ m=1,\ldots, n,
\)
and hence,
\[
\frac{1}{n} \sum\limits_{m=1}^n \omega_{\delta_m}(A_m)
\leq
\omega_{(f,\delta)}(\mathbb{A}).
\tag*{\qedhere} \]
\end{proof}

By Theorem \ref{thm4} and the subadditivity of the concave increasing function $ f^{-1} $ the following corollary holds.
\begin{corollary}
\label{cor5}
Let $  \mathbb{A} = (A_1, \ldots, A_n) \in \mathfrak{B}^{n}(\mathcal{H})
\ \text{and} \
0 \leq \delta_m \leq \Vert A_m \Vert,\; m=1,\ldots,n . $ Then,
\begin{align*}
f^{-1}\left( \sum\limits_{m=1}^{n}f\left(\omega_{\delta_{m}}(A_{m}) \right) \right) & \leq \sum\limits_{m=1}^{n} f^{-1}\left( f\left(\omega_{\delta_{m}}(A_{m}) \right) \right) \\
&= \sum\limits_{m=1}^{n}\omega_{\delta_{m}}(A_{m}) \leq n \omega_{(f,\delta)}(\mathbb{A}),
\end{align*}
where
\(
\delta = f^{-1}\left( \sum\limits_{m=1}^n f(\delta_m) \right).
\)
Hence,
\(
\sum\limits_{m=1}^{n}f\left( \omega_{\delta_{m}}(A_{m}) \right) \leq f\left(n \omega_{(f,\delta)}(\mathbb{A}) \right) .
\)
\end{corollary}

If we take $ f(t)=t^{2} $ and $ \delta_{m}=0, \ m=1,\ldots, n $ in Corollary \ref{cor5}, we obtain the following corollary.
\begin{corollary}
\label{cor6}
Let $  \mathbb{A} = (A_1, \ldots, A_n) \in \mathfrak{B}^{n}(\mathcal{H})$ be $n$-tuple of operators. Then
\[
\dfrac{1}{\sqrt{n}}\left( \sum\limits_{m=1}^{n}\omega^{2}(A_{m}) \right)^{1/2} \leq \omega (\mathbb{A}).
\]
and using the following well-known inequality
\(
\dfrac{1}{2}\Vert A_{m} \Vert \leq \omega (A_{m}), \ 1\leq m \leq n
\)
it follows that
\(
\dfrac{1}{2\sqrt{n}}\left( \sum\limits_{m=1}^{n}\Vert A_{m} \Vert^{2} \right)^{1/2} \leq \omega (\mathbb{A}).
\)

Also,
\(
\Vert \mathbb{A} \Vert \leq \left( \sum\limits_{m=1}^{n}\Vert A_{m} \Vert^{2} \right)^{1/2}
\)
yields
\(
\dfrac{1}{2\sqrt{n}} \Vert \mathbb{A} \Vert\leq \omega (\mathbb{A}),
\)
obtained in \textup{\cite{Muller2014}}.
\end{corollary}
\begin{theorem}
\label{thm5}
Let $  \mathbb{A} = (A_1, \ldots, A_n) \in \mathfrak{B}^{n}(\mathcal{H})$ be $n$-tuple of operators and $
\delta\leq \Vert \mathbb{A}  \Vert_{f}.$ Then
\[
f\left( \omega_{(f,\delta)}(\mathbb{A}) \right) \leq \sum\limits_{m=1}^{n}f\left( \omega_{\delta_{m}}(A_{m}) \right),
\]
where $ \delta_{m}=\delta-\tau_{m}, \ \tau_{m}=f^{-1}\left( \sum\limits_{\substack{k=1 \\ k\neq m}}^{n}f\left(  \Vert A_{k} \Vert\right)  \right) , \ m=1,\ldots, n. $
\end{theorem}
\begin{proof}  For any $ \delta\leq \Vert \mathbb{A} \Vert_{f}  $ and $ x\in \Delta_{(f,\delta)}(\mathbb{A}) $,
\begin{align*}
0 \leq \delta &\leq \Vert \mathbb{A} x \Vert_{f} = f^{-1}\left( \sum\limits_{m=1}^{n}f\left( \Vert A_{m} x\Vert \right)  \right)  \\
&\leq f^{-1}\left( f\left( \Vert A_{m} x\Vert \right)+ \left( \sum\limits_{\substack{k=1 \\ k\neq m}}^{n}f\left(  \Vert A_{k} \Vert\right)  \right) \right)  \\
&\leq \Vert A_{m} x\Vert + f^{-1}\left(\sum\limits_{\substack{k=1 \\ k\neq m}}^{n}f\left(  \Vert A_{k} \Vert\right)   \right) \\
&= \Vert A_{m} x\Vert + \tau_{m}, \ \tau_{m}=f^{-1}\left( \sum\limits_{\substack{k=1 \\ k\neq m}}^{n}f\left(  \Vert A_{k} \Vert\right)  \right) , \ m=1,\ldots, n.
\end{align*}
Hence,
\(
\delta-\tau_{m}\leq \Vert A_{m}x \Vert, \ m=1,\ldots,n.
\)
This means that for each $ m=1,\ldots,n $,
\(
x\in \Delta_{\delta_{m}}(A_{m}), \ \delta_{m}=\delta-\tau_{m}.
\)
Finally, it follows that
\[
\Delta_{(f,\delta)}(\mathbb{A})\subset\bigcap \limits_{m=1}^{n}\Delta_{\delta_{m}}(A_{m}).
\]
Thus, from
\(
f\left( \vert \langle \mathbb{A} x,x \rangle \vert_{f}  \right) = \sum\limits_{m=1}^{n} f\left( \vert \langle A_{m} x,x \rangle \vert  \right)
\)
it is clear that
\(
f\left( \omega_{(f,\delta)}(\mathbb{A}) \right) \leq \sum\limits_{m=1}^{n}f\left( \omega_{\delta_{m}}(A_{m}) \right).
\)
\end{proof}

If we take $ f(t)=t^{p}, \ 0\leq t <\infty, \ 1\leq p <\infty $ in Theorem \ref{thm5}, then we obtain the following corollary.
\begin{corollary}
\label{cor7}
Let \( \mathbb{A} = (A_1, \ldots, A_n) \in \mathfrak{B}^{n}(\mathcal{H})$ be $n$-tuple of operators and $ \ \delta\leq \Vert \mathbb{A} \Vert_{f} \). Then
\(
\omega_{(t^{p},\delta )}^{p}(\mathbb{A}) \leq \sum\limits_{m=1}^{n} \omega_{\delta_{m}}^{p}(A_{m}),
\)
where $ \delta_{m}=\delta -\tau_{m}, \ \tau_{m}=\left( \sum\limits_{\substack{k=1 \\ k\neq m}}^{n} \Vert A_{k} \Vert^{p}\right) ^{1/p}. $
Hence,
\(
\omega_{(f,\delta )}(\mathbb{A}) \leq \left( \sum\limits_{m=1}^{n} \Vert A_{m} \Vert^{p}\right)^{1/p} .
\)
\end{corollary}

If we take $ n=1 $ and $ \delta=0 $ in Corollary \ref{cor7}, then we obtain the following well-known inequality.
\begin{corollary}
\label{cor8}
Let $  A_1 \in \mathfrak{B}(\mathcal{H})$. Then,
\(
\omega (A_{1})\leq \Vert A_{1} \Vert.
\)
\end{corollary}

The next statement shows the relationship between the $(f,\delta)$-numerical radius of $n$-tuple of operators and that of its modulus.
\begin{theorem}
\label{thm6}
Let $ f $ be geometrically convex, $  \mathbb{A} = (A_1, \ldots, A_n) \in \mathfrak{B}^{n}(\mathcal{H}) $ be $n$-tuple of operators and   $
0\leq \delta\leq \Vert \mathbb{A}  \Vert_{f}.$ Then,
\[
f\left( \omega_{(f,\delta )}(\mathbb{A}  ) \right) \leq \sqrt{n}f^{1/2}\left( \Vert \ \vert \mathbb{A}  \vert \ \Vert \right) f^{1/2}\left( \omega_{(f,\delta)}(\vert \mathbb{A} \vert) \right).
\]
\end{theorem}
\begin{proof}
Since $ f $ geometrically convex, then by Lemma \ref{lem4} for any $ x\in S_{1}(\mathcal{H}) $
\begin{align}
\label{equ16}
f\left( \vert \langle \mathbb{A}x,x \rangle  \vert_{f}  \right)  & = \sum\limits_{m=1}^{n} f\left( \vert \langle A_{m}x,x \rangle  \vert  \right) \leq \sum\limits_{m=1}^{n}f\left( \Vert A_{m}x \Vert \right) \nonumber \\
&=\sum\limits_{m=1}^{n}f\left( \langle A_{m}^{*}A_{m}x,x\rangle^{1/2} \right)  =\sum\limits_{m=1}^{n}f\left(\Vert \ \vert A_{m}\vert x \Vert\right) \nonumber  \\
&\leq \sum\limits_{m=1}^{n}f\left(  \Vert \ \vert A_{m} \vert \ \Vert^{1/2}\langle \vert A_{m}\vert x,x\rangle^{1/2}\right) \nonumber \\
&= \sum\limits_{m=1}^{n} f^{1/2}\left(  \Vert \ \vert A_{m} \vert \ \Vert\right) f^{1/2}\left( \langle \vert A_{m}\vert x,x\rangle\right) \nonumber  \\
&\leq f^{1/2}\left( \Vert \ \vert \mathbb{A} \vert \ \Vert \right) \sum\limits_{m=1}^{n} f^{1/2}\left( \langle \vert A_{m}\vert x,x\rangle\right) \nonumber  \\
&\leq \sqrt{n} f^{1/2}\left( \Vert \ \vert \mathbb{A} \vert \ \Vert \right) \left( \sum\limits_{m=1}^{n} f\left( \vert \langle \vert A_{m}\vert x,x\rangle\vert\right)\right) ^{1/2} \nonumber  \\
&=\sqrt{n} f^{1/2}\left( \Vert \ \vert \mathbb{A} \vert \ \Vert \right) f^{1/2}\left( \vert \langle \vert \mathbb{A}\vert x,x\rangle \Vert_{f}\right).
\end{align}
On the other hand, since for $ x\in S_{1}(\mathcal{H}),$
\begin{multline*}
f(\delta) \leq f\left( \Vert \mathbb{A}x \Vert_{f} \right) = \sum\limits_{m=1}^{n} f\left( \Vert A_{m}x \Vert \right)
= \sum\limits_{m=1}^{n} f\left( \langle A_{m}^{*}A_{m}x,x\rangle^{1/2} \right) \\
 = \sum\limits_{m=1}^{n}f\left( \Vert \vert A_{m} \vert x  \Vert\right) = f\left( \Vert \ \vert \mathbb{A} \vert x \ \Vert_{f} \right) ,
\end{multline*}
it follows that
\begin{equation}
\label{equ17}
\Delta_{(f,\delta)}(\mathbb{A}) = \Delta_{(f,\delta)}( \vert \mathbb{A} \vert).
\end{equation}
Consequently, \eqref{equ16} and \eqref{equ17} imply
\[
f\left( \omega_{(f,\delta )}(\mathbb{A}  ) \right) \leq \sqrt{n}f^{1/2}\left( \Vert \ \vert \mathbb{A}\vert \ \Vert \right) f^{1/2}\left( \omega_{(f,\delta)}(\vert \mathbb{A} \vert) \right).
\tag*{\qedhere}\]
\end{proof}

Now, it will be given a result that established the relationship between the $(f,\delta)$-numerical radius of an operator and its real and imaginary parts.
\begin{theorem}
\label{thm7}
Let $ f $ be geometrically convex, $\mathbb{A}= (A_1, \ldots, A_n) \in \mathfrak{B}^{n}(\mathcal{H})$ be $n$-tuple of operators, and $0\leq \delta \leq \Vert \mathbb{A}\Vert_f$. Then
\[
f^{2}\left( \omega_{(f,\delta )}(\mathbb{A}  ) \right) \leq \sqrt{\dfrac{nf(1)}{2}}f\left(  \omega_{(f,\delta^{*})}\left( \mathbb{A}^{*}\mathbb{A}+\mathbb{A}\mathbb{A}^{*} \right) \right) ,
\]
where $ \delta^{*}=\dfrac{f^{2}(\delta)}{nf(1)}. $
\end{theorem}
\begin{proof}  Because $ f $ is geometrically convex, it is clear that for any $ x\in S_{1}(\mathcal{H}) $ and  $ m=1,\ldots,n $,
\begin{align*}
f\left(  \vert \langle A_{m}x,x\rangle \vert \right) &= f\left( \left( \vert \langle Re A_{m}x,x\rangle \vert ^{2} + \vert \langle ImA_{m}x,x\rangle \vert^{2}  \right)^{1/2}  \right) \\
&\leq f\left( \left( \Vert ReA_{m}x \Vert^{2}+ \Vert ImA_{m}x \Vert^{2}\right)^{1/2}   \right) \\
&= f\left( \langle \left( Re^{2}A_{m}+Im^{2}A_{m} \right)x,x\rangle^{1/2}  \right)  \\
&\leq f^{1/2}(1) f^{1/2}\left( \vert \langle \left( Re^{2}A_{m}+Im^{2}A_{m} \right)x,x  \rangle \vert \right).
\end{align*}
Hence, we obtain that
\begin{align*}
\sum\limits_{m=1}^{n}f\left( \vert \langle A_{m}x,x\rangle \vert \right)  &\leq f^{1/2}(1) \sum\limits_{m=1}^{n}f^{1/2}\left( \vert \langle \left( Re^{2}A_{m}+Im^{2}A_{m} \right)x,x  \rangle \vert \right) \\
&\leq f^{1/2}(1) n^{1/2} \left(\sum\limits_{m=1}^{n} f\left(\vert \langle \left( Re^{2}A_{m}+Im^{2}A_{m} \right)x,x\rangle  \vert \right)  \right) ^{1/2},
\end{align*}
that is,
\begin{align}
\label{equ18}
& f^{-1}\left( \sum\limits_{m=1}^{n}f\left( \vert \langle A_{m}x,x\rangle \vert \right)  \right)  \nonumber \\
&\leq f^{-1}\left( f^{1/2}(1) n^{1/2} \left(\sum\limits_{m=1}^{n} f\left(\vert \langle \left( Re^{2}A_{m}+Im^{2}A_{m} \right)x,x\rangle  \vert \right)  \right) ^{1/2} \right) \nonumber \\
&= f^{-1}\left( f^{1/2}(1) n^{1/2} \left(\sum\limits_{m=1}^{n} f\left(\vert \langle \left( \dfrac{A_{m}^{*}A_{m}+A_{m}A_{m}^{*}}{2} \right)x,x\rangle  \vert \right)  \right) ^{1/2} \right) \nonumber \\
&\leq f^{-1}\left( f^{1/2}(1) \sqrt{\dfrac{n}{2}} \left(\sum\limits_{m=1}^{n} f\left(\vert \langle \left( A_{m}^{*}A_{m}+A_{m}A_{m}^{*} \right)x,x\rangle  \vert \right)  \right) ^{1/2} \right)\nonumber \\
&= f^{-1}\left( f^{1/2}(1) \sqrt{\dfrac{n}{2}} \left( f\left(\vert \langle \left( \mathbb{A}^{*}\mathbb{A}+\mathbb{A}\mathbb{A}^{*} \right)x,x\rangle  \vert_{f} \right)  \right) ^{1/2} \right) \nonumber\\
&= f^{-1}\left( \sqrt{\dfrac{nf(1)}{2}}\sqrt{f\left(\vert \langle \left( \mathbb{A}^{*}\mathbb{A}+\mathbb{A}\mathbb{A}^{*} \right)x,x\rangle  \vert_{f} \right)} \right).
\end{align}
On the other hand, for each $ x\in S_{1}(H), $ since
\begin{align*}
f(\delta)&\leq f\left( \Vert \mathbb{A}x \Vert_{f} \right) = \sum\limits_{m=1}^{n} f\left( \Vert A_{m}x \Vert \right) \leq \sum\limits_{m=1}^{n}f\left( \Vert ReA_{m}x+iImA_{m}x \Vert \right)  \\
&\leq \sum\limits_{m=1}^{n} f\left(  \Vert ReA_{m}x \Vert + \Vert ImA_{m}x \Vert\right) \\
& \leq \sum\limits_{m=1}^{n} f\left(  \sqrt{2}\left( \Vert ReA_{m}x \Vert^{2} + \Vert ImA_{m}x \Vert^{2}\right)^{1/2} \right) \\
&= \sum\limits_{m=1}^{n} f\left(  \sqrt{2}\langle \left( Re^{2}A_{m}+ Im^{2}A_{m} \right)x,x\rangle^{1/2} \right) \\
&\leq \sum\limits_{m=1}^{n}f\left( \sqrt{2} \Vert \left( Re^{2}A_{m}+Im^{2}A_{m}\right) x \Vert^{1/2}\right) \\
&= \sum\limits_{m=1}^{n}f\left( \sqrt{2}\Vert \left( \dfrac{A_{m}^{*}A_{m}+A_{m}A_{m}^{*}}{2} \right) x \Vert^{1/2}  \right) \\
&= \sum\limits_{m=1}^{n}f\left( \Vert \left(  A_{m}^{*}A_{m}+A_{m}A_{m}^{*}\right) x \Vert^{1/2} \right)   \\
&\leq \sum\limits_{m=1}^{n}f^{1/2}(1) f^{1/2} \left( \Vert \left(  A_{m}^{*}A_{m}+A_{m}A_{m}^{*}\right) x \Vert \right) \\
&= f^{1/2}(1)\sqrt{n} \left( \sum\limits_{m=1}^{n}f\left( \Vert \left(  A_{m}^{*}A_{m}+A_{m}A_{m}^{*}\right) x \Vert \right)  \right)^{1/2} \\
&= f^{1/2}(1)\sqrt{n} \left( f \left( \Vert \left(  \mathbb{A}^{*}\mathbb{A}+\mathbb{A}\mathbb{A}^{*}\right) x \Vert_{f} \right)  \right) ^{1/2},
\end{align*}
it follows that
\(
f^{-1} \left(\dfrac{f^{2}(\delta)}{nf(1)}\right)  \leq \Vert \left(  \mathbb{A}^{*}\mathbb{A}+\mathbb{A}\mathbb{A}^{*}\right) x \Vert_{f}  .
\)
Hence,
\begin{equation}
\label{equ19}
\Delta_{(f,\delta )}(\mathbb{A})\subset \Delta_{(f,\delta^{*})}(\mathbb{A}^{*}\mathbb{A}+\mathbb{A}^{*}\mathbb{A})
\end{equation}
where $ \delta^{*}=f^{-1}\left( \dfrac{f^{2}(\delta)}{nf(1)}\right) . $
Therefore, from \eqref{equ18} and \eqref{equ19} it follows that
\[
\omega_{(f,\delta)}(\mathbb{A})\leq f^{-1}\left( \sqrt{\dfrac{nf(1)}{2}}\sqrt{f\left( \omega_{(f,\delta^{*})}(\mathbb{A}^{*}\mathbb{A}+\mathbb{A}^{*}\mathbb{A})\right) } \right) ,
\]
that is,
\(
f^{2}\left( \omega_{(f,\delta )}(\mathbb{A}  ) \right) \leq \sqrt{\dfrac{nf(1)}{2}}f\left(  \omega_{(f,\delta^{*})}\left( \mathbb{A}^{*}\mathbb{A}+\mathbb{A}\mathbb{A}^{*} \right) \right) .
\)
\end{proof}

If we take $ f(t)=t^{p}, \ 0\leq t<\infty, \ 1\leq p< \infty $ in Theorem \ref{thm7}, we obtain the following corollary.
\begin{corollary}
\label{cor9}
Let $\mathbb{A} \in \mathfrak{B}^{n}(\mathcal{H})$ be $n$-tuple of operators and $0\leq \delta \leq \Vert \mathbb{A}\Vert_f.$ Then,
\(
\omega_{(t^{2},\delta )}^{2p}(\mathbb{A}  ) \leq \dfrac{n}{2}\omega_{(t^{2},\delta^{*})}^{p}\left( \mathbb{A}^{*}\mathbb{A}+\mathbb{A}\mathbb{A}^{*} \right),
\)
where $ \delta^{*}=\dfrac{\delta^{2}}{n^{1/p}}  . $
\end{corollary}

If we take $ n=1 $ and $ \delta=0 $ in Corollary \ref{cor9}, then we obtain the following corollary, which states the second inequality in \eqref{equ2}.
\begin{corollary}
\label{cor10}
Let $A_{1} \in \mathfrak{B}(\mathcal{H})$. Then
\(
\omega^{2}(A_{1})\leq \dfrac{1}{2} \Vert A_{1}^{*}A_{1}+A_{1}A_{1}^{*} \Vert .
\)
\end{corollary}
\begin{theorem}
\label{thm8}
Let $ \mathbb{A}\in\mathfrak{B}^{n}(\mathcal{H}) $ be $n$-tuple of operators. Then
\begin{align*}
\omega_{(f,\delta)}(Re\mathbb{A}) & \leq \omega _{(f,0)}(\mathbb{A})  \ \text{for} \ \delta\leq \Vert Re\mathbb{A} \Vert_{f}\ ,
\\
\omega_{(f,\delta)}(Im\mathbb{A}) & \leq \omega _{(f,0)}(\mathbb{A})  \ \text{for} \  \delta\leq \Vert Im\mathbb{A} \Vert_{f}\ .
\end{align*}
\end{theorem}
\begin{proof}
For any $ x\in \mathcal{H}, $ it is clear that
\begin{multline*}
\vert \langle (Re\mathbb{A})x,x\rangle \vert = f^{-1}\left( \sum\limits_{m=1}^{n}f\left( \vert \langle ReA_{m}x,x\rangle \vert \right)   \right)\\
\leq f^{-1}\left( \sum\limits_{m=1}^{n}f\left( \vert \langle A_{m}x,x\rangle \vert \right)  \right)
= \vert \langle \mathbb{A}x,x \rangle \vert_{f}
\end{multline*}
and similarly
\(
\vert \langle Im\mathbb{A}x,x\rangle \vert_{f} \leq \vert \langle \mathbb{A}x,x \rangle \vert_{f} .
\)
It is easily to see that
\begin{align*}
& \Delta_{(f,\delta)}(Re\mathbb{A})\subset S_{1}(H) \ \text{for any} \ \delta \leq \Vert Re\mathbb{A} \Vert_{f},
\\
& \Delta_{(f,\delta)}(Im\mathbb{A})\subset S_{1}(H) \ \text{for any} \ \delta \leq \Vert Im\mathbb{A} \Vert_{f}.
\end{align*}
Therefore, from these relations it follows that
\begin{align*}
&\omega_{(f,\delta)}(Re\mathbb{A})\leq \omega _{(f,0)}(\mathbb{A})  \ \text{for} \ \delta\leq \Vert Re\mathbb{A} \Vert_{f},
\\
&\omega_{(f,\delta)}(Im\mathbb{A})\leq \omega _{(f,0)}(\mathbb{A})  \ \text{for} \  \delta\leq \Vert Im\mathbb{A} \Vert_{f} . \tag*{\qedhere}
\end{align*}
\end{proof}
\begin{remark}
Theorem \ref{thm8} generalizes \textnormal{\cite[Thm. 2.9]{Ismailov2026}}.
\end{remark}
\section{Sectorial case of coordinate operators}
\label{Sec:3}
In this section, the joint $(f,\delta)$-numerical radius is investigated in the case when the coordinate operators of $n$-tuple of operators are sectoral.
\begin{theorem}
\label{thm9}
Let $\mathbb{A} = (A_{1}, \ldots, A_{n}) \in \mathfrak{B}^{n}(\mathcal{H})$ and $ 0\leq\delta\leq \Vert \mathbb{A} \Vert_{f}. $
\begin{enumerate}[label=\upshape{\arabic*.}, ref=\upshape{\arabic*}, labelindent=5pt, leftmargin=*,ref=\upshape{\arabic*.}]
\item \label{item1-thm9} If $ A_{m}\in Sec_{\alpha_{m}}(\mathcal{H}), \ 0\leq \alpha_{m}<\dfrac{\pi}{2}, \ m=1,\ldots ,n, $ then
\[
\omega_{(f,\delta)}(\mathbb{A})\leq \omega_{(f,\delta_{r})}(\sec(\alpha_{max})Re\mathbb{A}),
\]
where $ \alpha_{max}=\max\limits_{1\leq m \leq n}\alpha_{m} $ and $ \delta_{r}=\delta - \Vert Im \mathbb{A} \Vert_{f} .$
\item \label{item2-thm9} If $ A_{m}\in Sec_{\alpha_{m}}(\mathcal{H}), \ 0< \alpha_{m}<\dfrac{\pi}{2}, \ m=1,\ldots ,n, $ then
\[
\omega_{(f,\delta)}(\mathbb{A})\leq \omega_{(f,\delta_{i})}(\csc(\alpha_{min})Im\mathbb{A}),
\]
where $ \alpha_{min}=\min\limits_{1\leq m \leq n}\alpha_{m} $ and $ \delta_{i}=\delta - \Vert Re \mathbb{A} \Vert_{f} .$
\end{enumerate}
\end{theorem}
\begin{proof}
 \noindent \ref{item1-thm9}\ For any $ x\in \mathcal{H},
 $ it is clear that
\begin{align*}
f\left(  \vert \langle \mathbb{A}x,x \rangle \vert_{f}\right) &= \sum\limits_{m=1}^{n} f\left(  \vert \langle A_{m}x,x \rangle \vert\right) \\
&= \sum\limits_{m=1}^{n} f\left( \left( \vert \langle ReA_{m}x,x\rangle \vert^{2} + \vert \langle ImA_{m}x,x\rangle \vert^{2}\right) ^{1/2} \right) \\
&\leq \sum\limits_{m=1}^{n} f \left( \left( 1+\tan^{2}(\alpha_{m}) \right)^{1/2} \vert \langle ReA_{m}x,x\rangle \vert \right) \\
&\leq \sum\limits_{m=1}^{n} f \left( \sec(\alpha_{max})\vert \langle ReA_{m}x,x\rangle \vert \right) \\
&= \sum\limits_{m=1}^{n} f \left( \vert \langle \sec(\alpha_{max})ReA_{m}x,x\rangle \vert \right) \\
&= f \left( \vert \langle \sec(\alpha_{max})Re\mathbb{A}x,x\rangle \vert_{f}\right).
\end{align*}
Consequently, for $ x\in \mathcal{H}$
\begin{equation}
\label{equ20}
\vert \langle \mathbb{A} x,x \rangle \vert_{f}\leq \vert \langle \sec(\alpha_{max})Re\mathbb{A}x,x\rangle \vert_{f}.
\end{equation}
On the other hand, for $ 0\leq\delta\leq \Vert  \mathbb{A} \Vert_{f} $ and any $ x\in \Delta_{(f,\delta)}( \mathbb{A}) $ we obtain
\begin{align*}
\delta &\leq \Vert \mathbb{A} x \Vert_{f} = f^{-1}\left( \sum\limits_{m=1}^{n}f\left( \Vert A_{m}x \Vert \right)  \right) \\
&\leq f^{-1} \left( \sum\limits_{m=1}^{n}f\left( \Vert ReA_{m}x \Vert + \Vert Im A_{m}x \Vert \right)  \right) \\
&\leq f^{-1} \left( \sum\limits_{m=1}^{n}f\left( \Vert ReA_{m}x \Vert \right)  \right)+f^{-1} \left( \sum\limits_{m=1}^{n}f\left( \Vert Im A_{m}x \Vert \right)  \right) \\
&= f^{-1} \left( \sum\limits_{m=1}^{n}f\left( \cos(\alpha_{max})\Vert \sec(\alpha_{max})ReA_{m}x \Vert \right)  \right)+ \Vert Im \mathbb{A} \Vert_{f}\\
&= f^{-1}\left( \cos(\alpha_{max})\sum\limits_{m=1}^{n}f\left( \Vert \sec(\alpha_{max})ReA_{m}x \Vert \right)  \right) + \Vert Im \mathbb{A} \Vert_{f} \\
&\leq f^{-1}\left(\sum\limits_{m=1}^{n}f\left( \Vert \sec(\alpha_{max})ReA_{m}x \Vert \right)  \right) + \Vert Im \mathbb{A} \Vert_{f}.
\end{align*}
Then, it is clear that
\(
\delta - \Vert Im \mathbb{A} \Vert_{f} \leq  \Vert \sec(\alpha_{max})Re\mathbb{A}x \Vert_{f},
\)
and setting $\delta_{r}=\delta - \Vert Im \mathbb{A} \Vert_{f},   $
yields
\(
\delta_{r} \leq  \Vert \sec(\alpha_{max})Re\mathbb{A}x \Vert_{f}.
\)
Therefore,
\begin{equation}
\label{equ21}
\Delta_{(f,\delta)}(\mathbb{A})\subset\Delta_{(f,\delta_{r})}\left( \sec(\alpha_{max})Re\mathbb{A} \right) .
\end{equation}
Consequently, from the \eqref{equ20} and \eqref{equ21}, it follows that
\[
\omega_{(f,\delta)}(\mathbb{A})\leq \omega_{(f,\delta_{r})}(\sec(\alpha_{max})Re\mathbb{A}).
\]
\noindent \ref{item2-thm9} For any $ 1\leq m\leq n $ and $ x\in \mathcal{H} $,
\[
\vert \langle Re A_{m}x,x\rangle \vert \leq \cot(\alpha_{m}) \vert \langle Im A_{m}x,x\rangle \vert .
\]
Therefore, for $ x\in \mathcal{H} $ we obtain
\begin{align*}
f\left(  \vert \langle \mathbb{A}x,x \rangle \vert_{f}\right) &= \sum\limits_{m=1}^{n} f\left(  \vert \langle A_{m}x,x \rangle \vert\right) \\
&= \sum\limits_{m=1}^{n} f\left( \left( \vert \langle ReA_{m}x,x\rangle \vert^{2} + \vert \langle ImA_{m}x,x\rangle \vert^{2}\right) ^{1/2} \right) \\
&\leq \sum\limits_{m=1}^{n} f \left( \left( 1+\cot^{2}(\alpha_{m}) \right)^{1/2} \vert \langle ImA_{m}x,x\rangle \vert )^{1/2}\right) \\
&\leq \sum\limits_{m=1}^{n} f \left( \csc(\alpha_{min})\vert \langle ImA_{m}x,x\rangle \vert \right) \\
&= \sum\limits_{m=1}^{n} f \left( \vert \langle \csc(\alpha_{min})ImA_{m}x,x\rangle \vert \right) \\
&= f \left( \vert \langle \csc(\alpha_{min})Im\mathbb{A}x,x\rangle \Vert_{f}\right) .
\end{align*}
Consequently, for $ x\in \mathcal{H} $
\begin{equation}
\label{equ22}
\vert \langle \mathbb{A} x,x \rangle \vert_{f}\leq \vert \langle \csc(\alpha_{min})Im\mathbb{A}x,x\rangle \Vert_{f} .
\end{equation}
On the other hand, for $ 0\leq\delta\leq \Vert  \mathbb{A} \Vert_{f} $ and any $ x\in \Delta_{(f,\delta)}( \mathbb{A}) $ we obtain
\begin{align*}
\delta &\leq \Vert \mathbb{A} x \Vert_{f} = f^{-1}\left( \sum\limits_{m=1}^{n}f\left( \Vert A_{m}x \Vert \right)  \right) \\
&\leq f^{-1} \left( \sum\limits_{m=1}^{n}f\left( \Vert ReA_{m}x \Vert + \Vert Im A_{m}x \Vert \right)  \right) \\
&\leq f^{-1} \left( \sum\limits_{m=1}^{n}f\left( \Vert ReA_{m}x \Vert \right)  \right)+f^{-1} \left( \sum\limits_{m=1}^{n}f\left( \Vert Im A_{m}x \Vert \right)  \right) \\
&\leq f^{-1}\left( \sum\limits_{m=1}^{n}f\left(\sin(\alpha_{min})\langle \csc(\alpha_{min})Im A_{m}x,x\rangle \right)  \right)\\
&\hspace{5cm} +f^{-1} \left( \sum\limits_{m=1}^{n}f\left( \Vert ReA_{m}x \Vert \right)  \right) \\
&\leq f^{-1} \left( \sum\limits_{m=1}^{n}\sin(\alpha_{min})f\left( \Vert \csc(\alpha_{min})ImA_{m}x \Vert \right)  \right)+ \Vert Re \mathbb{A} \Vert_{f}\\
&\leq f^{-1}\left( \sin(\alpha_{min})\sum\limits_{m=1}^{n}f\left( \Vert \csc(\alpha_{min})ImA_{m}x \Vert \right)  \right) + \Vert Re \mathbb{A} \Vert_{f}\\
&\leq f^{-1}\left(\sum\limits_{m=1}^{n}f\left( \Vert \csc(\alpha_{min})ImA_{m}x \Vert \right)  \right) + \Vert Re \mathbb{A} \Vert_{f}\\
\end{align*}
Then, it is clear that
\(
\delta - \Vert Re \mathbb{A} \Vert_{f} \leq  \Vert \csc(\alpha_{min})Im\mathbb{A}x \Vert_{f},
\)
and setting $\delta_{i}=\delta - \Vert Re \mathbb{A} \Vert_{f},   $
we obtain that
\(
\delta_{i} \leq  \Vert \csc(\alpha_{min})Im\mathbb{A}x \Vert_{f}.
\)
Therefore,
\begin{equation}
\label{equ23}
\Delta_{(f,\delta)}(\mathbb{A})\subset\Delta_{(f,\delta_{i})}\left( \csc(\alpha_{min})Im\mathbb{A} \right) .
\end{equation}
Consequently, \eqref{equ22} and \eqref{equ23} yield
\(
\omega_{(f,\delta)}(\mathbb{A})\leq \omega_{(f,\delta_{i})}(\csc(\alpha_{min})Im\mathbb{A}).
\)
 \end{proof}

If we take $ f(t)=t^{2}, \ 0\leq t < \infty  $  and $ \delta=0 $ in Theorem \ref{thm9}, we obtain the following corollary, in which the statement \label{item1-cor7} was previously proved in \cite{Bedrani2020}.
\begin{corollary}
\label{cor11}
Let  $ \mathbb{A}=(A_{1},\ldots , A_{n})\in\mathfrak{B}^{n}(\mathcal{H}). $
\begin{enumerate}[label=\upshape{\arabic*.}, ref=\upshape{\arabic*}, labelindent=5pt, leftmargin=*,ref=\upshape{\arabic*}]
\item \label{item1-cor11} If $ A_{m}\in Sec_{\alpha_{m}}(\mathcal{H}), \ 0\leq \alpha_{m}< \frac{\pi}{2}, \ m=1, \ldots ,n,  $ then
\[
\omega (\mathbb{A})\leq \sec(\alpha_{max})\omega (Re \mathbb{A}),
\]
where $ \alpha_{max}=\max\limits_{1\leq m \leq n}\alpha_{m}. $
Additionally, if $ n=1, $ then
\[
\omega (A_{1})\leq \sec(\alpha_{1})\Vert ReA_{1} \Vert .
\]
\item \label{item2-cor11} If $ A_{m}\in Sec_{\alpha_{m}}(\mathcal{H}), \ 0 < \alpha_{m}< \frac{\pi}{2}, \ m=1, \ldots ,n, $ then
\[
\omega (\mathbb{A})\leq \csc(\alpha_{min})\omega (Im \mathbb{A}),
\]
where $ \alpha_{min}=\min\limits_{1\leq m \leq n}\alpha_{m}. $
Additionally, if $ n=1, $ then
\[
\omega (A_{1})\leq \sec(\alpha_{1})\Vert ImA_{1} \Vert .
\]
\end{enumerate}
\end{corollary}
\begin{remark}
In a special case, Corollary \ref{cor11} generalizes \textnormal{\cite[Corollary 3.3]{Ismailov2026}}.
\end{remark}
\begin{remark}
In a special case, Theorem \ref{thm9} generalizes \textnormal{\cite[Thm. 3.1]{Ismailov2026}}.
\end{remark}
\begin{theorem}
\label{thm10}
Let $ \mathbb{A}=(A_{1},\ldots , A_{n})\in\mathfrak{B}^{n}(\mathcal{H}) $ be $n$-tuple of operators and for each $ m=1,\ldots ,n $ the coordinate operator $ A_{m} $ is accretive-dissipative operator such that
\(
W(A_{m})\subset \left\lbrace  re^{i\varphi}: r\geq 0, \ 0\leq \gamma_{m}\leq \varphi \leq \alpha_{m}<\dfrac{\pi}{2}\right\rbrace .
\)
\begin{enumerate}[label=\upshape{\arabic*.}, ref=\upshape{\arabic*}, labelindent=5pt, leftmargin=*,ref=\upshape{\arabic*.}]
\item \label{item1-thm10} For any $ \delta\leq \Vert \sec(\gamma_{min})Re \mathbb{A} \Vert_{f}, $
\[
\omega_{(f,\delta)}(\sec(\gamma_{min})Re \mathbb{A})\leq \omega_{(f,0)}(\mathbb{A}).
\]
\item  \label{item2-thm10} For any $ \delta\leq \Vert \csc(\alpha_{max})Im \mathbb{A} \Vert_{f}, $
\[
\omega_{(f,\delta)}(\csc(\alpha_{max})Im \mathbb{A})\leq \omega_{(f,0)}(\mathbb{A}).
\]
\end{enumerate}
\end{theorem}
\begin{proof}
\noindent \ref{item1-thm10}\ For any $ x\in \mathcal{H} $ we have
\begin{multline*}
f\left( \vert \langle \mathbb{A} x,x \rangle \vert_{f} \right) = \sum\limits_{m=1}^{n} f\left( \vert \langle A_{m} x,x \rangle \vert_{f} \right) \\
= \sum\limits_{m=1}^{n}f\left( \left( \vert \langle ReA_{m} x,x \rangle \vert^{2} + \vert \langle ImA_{m} x,x \rangle \vert^{2} \right)^{1/2} \right).
\end{multline*}
On the other hand, since for $ x\in \mathcal{H} $
\[
\vert \langle ImA_{m}x,x\rangle \vert \geq \tan(\gamma_{m}) \vert \langle ReA_{m}x,x\rangle \vert , \ m=1,\ldots, n
\]
and the tangent function is monotonically increasing $ \left[ 0, {\pi/2}\right]$, then from the previous estimate one gets
\begin{align*}
f\left( \vert \langle \mathbb{A} x,x \rangle \vert_{f} \right) &= \sum\limits_{m=1}^{n}f\left( \left(1+\tan(\gamma_{m}) \right)^{1/2}  \right) \vert \langle ReA_{m} x,x \rangle \vert \\
&\geq \sum\limits_{m=1}^{n}f\left( \left(1+\tan(\gamma_{min}) \right)^{1/2}  \vert \langle ReA_{m} x,x \rangle \vert \right) \\
&= \sum\limits_{m=1}^{n}f\left( \sec(\gamma_{min}) \vert \langle ReA_{m} x,x \rangle \vert \right)  \\
&= \sum\limits_{m=1}^{n}f\left( \vert \langle \sec(\gamma_{min}) ReA_{m} x,x \rangle \vert \right)  \\
&= f\left( \vert \langle \sec(\gamma_{min}) Re\mathbb{A} x,x \rangle \vert \right).
\end{align*}
Thus, for any $ x\in \mathcal{H} $,
\begin{equation}
\label{equ24}
\vert \langle \mathbb{A} x,x \rangle \vert_{f}\geq \vert \langle \sec(\gamma_{min})\mathbb{A} x,x \rangle \vert_{f},
\end{equation}
where $ \delta_{min}=\min\limits_{1\leq m \leq n}\delta_{m}. $
For any $ \delta\leq \Vert Re \mathbb{A} \Vert_{f} $ it is clear that
\begin{equation}
\label{equ25}
\Delta_{(f,\delta)} (\sec(\gamma_{min})Re\mathbb{A})\subset S_{1}(\mathcal{H}).
\end{equation}
Hence, from \eqref{equ24} and \eqref{equ25} we obtain that
\(
\omega_{(f,\delta)}(\sec(\gamma_{min})Re \mathbb{A})\leq \omega_{(f,0)}(\mathbb{A}).
\)
\noindent \ref{item2-thm10} From the following equality, for $ x\in \mathcal{H} $,
\[
f\left( \vert \langle \mathbb{A} x,x \rangle \vert_{f} \right) = \sum\limits_{m=1}^{n}f\left( \left( \vert \langle ReA_{m} x,x \rangle \vert^{2} + \vert \langle ImA_{m} x,x \rangle \vert^{2} \right)^{1/2} \right),
\]
together with the inequality
\(
\vert \langle ImA_{m} x,x \rangle \vert \leq \tan(\alpha_{m}) \vert \langle ReA_{m} x,x \rangle \vert,
\)
and by the decreasing monotonicity of cotangent function on $ \left( 0, {\pi/2} \right), $
\begin{multline*}
f\left( \vert \langle \mathbb{A} x,x \rangle \vert_{f} \right) \geq \sum\limits_{m=1}^{n}f\left( \left(1+\cot^{2}(\alpha_{m}) \right)^{1/2}  \right) \vert \langle ImA_{m} x,x \rangle \vert \\
= \sum\limits_{m=1}^{n}f\left( \csc(\alpha_{m}) \vert \langle ImA_{m} x,x \rangle \vert \right) = \sum\limits_{m=1}^{n}f\left( \vert \langle \csc(\alpha_{max}) ImA_{m} x,x \rangle \vert \right)  \\
= f\left(  \langle \csc(\alpha_{max}) Im \mathbb{A} x,x \rangle \right).
\end{multline*}
Hence, for $ x\in \mathcal{H} $
\begin{equation}
\label{equ26}
\vert \langle \csc(\alpha_{max}) Im\mathbb{A} x,x \rangle \Vert_{f}\leq \vert \langle \mathbb{A} x,x  \rangle \Vert_{f} .
\end{equation}
On the other hand, for any $ \delta \leq \Vert  \csc(\alpha_{max}) Im\mathbb{A}  \Vert_{f} $ it is clear that
\begin{equation}
\label{equ27}
\Delta_{(f,\delta)}(\csc(\alpha_{max}) Im\mathbb{A})\subset S_{1}(H).
\end{equation}
Consequently, from \eqref{equ26} and \eqref{equ27} it is established that
\[
\omega_{(f,\delta)}(\csc(\alpha_{max})Im \mathbb{A})\leq \omega_{(f,0)}(\mathbb{A}).  \tag*{\qedhere}
\]
\end{proof}
\begin{remark}
Theorem \ref{thm10} generalizes \textnormal{\cite[Thm. 3.5]{Ismailov2026}}.
\end{remark}
\begin{corollary}
\label{cor12}
Under the conditions of Theorem \ref{thm10}, if $ n=1, \ \delta=0 $ and $ 0\leq \gamma_{1} <\alpha_{1} < \frac{\pi}{2}  $, it follows that
\[
\omega(ReA_{1})\leq \cos(\gamma_{1}) \omega(A_{1}) \ \text{and} \ \omega(ImA_{1})\leq \sin(\gamma_{1}) \omega(A_{1}),
\]
and thus,
\begin{align*}
\omega^{2}(A_{1}) & \geq \dfrac{\Vert ReA_{1} \Vert^{2}+\Vert ImA_{1} \Vert^{2}} {\sin^{2}(\alpha_{1})+\cos^{2}(\gamma_{1})} \\
& \geq \dfrac{\Vert ReA_{1} + ImA_{1} \Vert} {\sin^{2}(\alpha_{1})+\cos^{2}(\gamma_{1})}
= \dfrac{\Vert A_{1}^{*}A_{1}+A_{1}A_{1}^{*}\Vert} {2\left( \sin^{2}(\alpha_{1})+\cos^{2}(\gamma_{1})\right) } .
\end{align*}
\end{corollary}
It must be noted that for an accretive operator $ A_{1}\in\mathfrak{B}(\mathcal{H})  $ satisfying the condition $ W(A_{1})\subset S(\alpha_{1}) ,$  the following evaluation was proved in \cite{Sammour2022}
\[
\dfrac{\Vert A_{1}^{*}A_{1}+A_{1}A_{1}^{*}\Vert} {2\left( 1+\sin^{2}(\alpha_{1}))\right) } \leq \omega ^{2}(A_{1}).
\]
However, the previous estimate shows that the last result can be improved for accretive-dissipative matrices.
\begin{corollary}
\label{cor13}
Let $ \dim\mathcal{H}<\infty. $ Under the condition of Theorem \ref{thm10} if $ n=1, \ \delta=0 $ and $ 0\leq \gamma_{1} <\alpha_{1} < \frac{\pi}{2}  $, then
\[
\Vert ReA_{1}\Vert\leq \cos(\gamma_{1}) \omega(A_{1}) \ \text{and} \ \Vert ImA_{1}\Vert\leq \sin(\alpha_{1}) \omega(A_{1}).
\]
Then, it is well-known from \textnormal{\cite{Sammour2022}} that
\[
\Vert A_{1} \Vert^{2} \leq \Vert ReA_{1} \Vert^{2} + \Vert ImA_{1} \Vert^{2} \leq \left( \sin^{2}(\alpha_{1})+\cos^{2}(\gamma_{1}) \right) \omega^{2}(A_{1}).
\]
Hence,
\(
\Vert A_{1} \Vert\leq \sqrt{\sin^{2}(\alpha_{1})+\cos^{2}(\gamma_{1})} \  \omega(A_{1}).
\)
\end{corollary}

It must be noted that in the matrix case of $ A_{1}\in Sec_{\alpha_{1}}(\mathcal{H}), \ 0\leq \alpha_{1} <\dfrac{\pi}{2}, $ the following main result was obtained in \cite{Sammour2022},
\[
\Vert A_{1} \Vert\leq \sqrt{1+\sin^{2}(\alpha_{1})}\ \omega(A_{1}).
\]
Therefore, the previous estimation shows that in the accretive-dissipative matrix case, the last inequality can be improved.

It must be noted that for accretive-dissipative operators $ A_{1},B_{1}\in \mathfrak{B}(\mathcal{H}) $ with conditions
\begin{align*}
W(A_{1}) & \subset \left\lbrace  re^{i\varphi}: r\geq 0, \ 0< \gamma_{1}<\alpha_{1}<\dfrac{\pi}{2}\right\rbrace
\\
W(B_{1}) & \subset \left\lbrace  re^{i\varphi}: r\geq 0, \ 0< \gamma_{1}<\alpha_{1}<\dfrac{\pi}{2}\right\rbrace
\end{align*}
from the inequalities
\begin{align*}
\Vert A_{1} \Vert & \leq \sqrt{\sin^{2}(\alpha_{1})+\cos^{2}(\gamma_{1})}\omega(A_{1}),
\\
\Vert B_{1} \Vert & \leq \sqrt{\sin^{2}(\alpha_{1})+\cos^{2}(\gamma_{1})}\omega(B_{1}),
\end{align*}
we have the following interpolated inequality for the matrices $ A_{1} $ and $ B_{1} $
\[
\omega(A_{1}B_{1})\leq \left( \sin^{2}(\alpha_{1})+\cos^{2}(\gamma_{1}) \right) \omega(A_{1}) \omega(B_{1}).
\]
In this case, from the last relation it is clear that
\[
\omega(A_{1}B_{1})\leq \left( 1+\sin^{2}(\alpha_{1}) \right) \omega(A_{1}) \omega(B_{1}).
\]
This was proved for accretive-dissipative matrices in \cite{Sammour2022}. However, the previous evaluation shows that the last interpolated formula can be improved.

\end{document}